\documentclass[preprint,12pt]{elsarticle}
\usepackage{amssymb}
\usepackage{amsthm}
\usepackage{mathrsfs}
\usepackage{enumerate}
\usepackage{amsmath}
\usepackage{amsfonts}
\usepackage{amssymb}
\usepackage{graphicx}
\usepackage{titlesec}
\usepackage{etoolbox}
\makeatletter
\patchcmd{\ttlh@hang}{\parindent\z@}{\parindent\z@\leavevmode}{}{}
\patchcmd{\ttlh@hang}{\noindent}{}{}{}
\makeatother
\usepackage{titletoc}
\usepackage{fancyhdr}
\usepackage{indentfirst}
\usepackage[margin=0.8in,vmargin=2.7cm]{geometry}

\numberwithin{equation}{section}

\newtheorem{theo}{Theorem}[section]
\newtheorem{lem}[theo]{Lemma}
\newtheorem{prop}[theo]{Proposition}
\newtheorem{cor}[theo]{Corollary}
\newtheorem{rem}[theo]{Remark}
\newtheorem{defn}[theo]{Definition}

\begin{document}

\begin{frontmatter}



\title{Eulerian droplet model: Delta-shock waves and solution of the Riemann problem}

\author[Author1]{Sana Keita}    \ead{skeit085@uottawa.ca}
\author[Author1]{Yves Bourgault\corref{cor1}} \ead{ybourg@uottawa.ca}
\cortext[cor1]{Corresponding author.}
\address[Author1]{Department of Mathematics and Statistics, University of Ottawa,\\ K1N 6N5 , Ottawa, Ontario, Canada}

\begin{abstract}
We study an Eulerian droplet model which can be seen as the pressureless gas system with a source term, a subsystem of this model and the inviscid Burgers equation with source term. The condition for loss of regularity of a solution to Burgers equation with source term is established. The same condition applies to the Eulerian droplet model and its subsystem. The Riemann problem for the Eulerian droplet model is constructively solved by going through the solution of the Riemann problems for the inviscid Burgers equation with a source term and the subsystem, respectively. Under suitable generalized Rankine-Hugoniot relations and entropy condition, the existence  of delta-shock solution is  established. The existence of a solution to the generalized Rankine-Hugoniot conditions is proven. Some numerical illustrations are presented.
\end{abstract}


\end{frontmatter}

\section{Introduction}
\label{SectEulDropMod}
In this paper we consider the one-dimensional Eulerian droplet model \cite{Bourgault1} in conservative form 
\begin{equation}
\left\lbrace
\begin{aligned}
&\partial_t\alpha+\partial_x(\alpha u)=0,\\
&\partial_t(\alpha u)+\partial_x(\alpha u^2)=\mu\alpha(u_a-u),
\end{aligned}
\right.
%\tag{E}
\label{DropModSimp}
\end{equation}
where $\alpha$ and $u$ denote, the volume fraction and velocity of the particles (droplets), respectively, $u_a$ is the velocity of the carrier fluid (air), and $\mu$ is the drag coefficient between the carrier fluid  and the particles. Since the density of particles exceeds the air density by orders of magnitude, the virtual mass force is neglected. The lift force, gravity, and other interfacial effects are also negligible when compared to the viscous drag force. These forces could be important in other applications \cite{Bourgault1}.  The Eulerian droplet model \eqref{DropModSimp} corresponds to a dispersed phase subsystem in its simplest form, for instance a multi-phase system for particles suspended in a carrier fluid.  
For smooth solutions, the second equation of \eqref{DropModSimp} is equivalent to
\begin{equation}
\alpha(\partial_tu+u\partial_xu)+u(\partial_t\alpha+\partial_x(\alpha u))=\mu\alpha(u_a-u).
\label{EquivUSmoothSol1}
\end{equation}
Using the first equation of \eqref{DropModSimp} and  simplifying by $\alpha\neq0$, \eqref{EquivUSmoothSol1} reduces to
\begin{equation}
\partial_tu+u\partial_x u=\mu(u_a-u)
\label{EquivUSmoothSol2}
\end{equation}
which can be rewritten in conservative form as
\begin{equation}
\partial_tu+\partial_x(\frac{1}{2}u^2)=\mu(u_a-u).
\label{BurgersSourceTerm}
\end{equation}
Hence, for smooth solutions with $\alpha\neq0$, the Eulerian droplet model \eqref{DropModSimp} is equivalent to
\begin{equation}
\left\lbrace
\begin{aligned}
&\partial_t\alpha+\partial_x(\alpha u)=0,\\
&\partial_tu+\partial_x(\frac{1}{2}u^2)=\mu(u_a-u),
\end{aligned}
\right.
%\tag{E'}
\label{DropModSimpAlphaU}
\end{equation}
One easily shows that any smooth solution of \eqref{DropModSimpAlphaU} is also a solution to \eqref{DropModSimp}. In the following, equation \eqref{BurgersSourceTerm} will be also referred to as the \textit{inviscid Burgers equation with source term}.\\

If $\mu=0$ then \eqref{BurgersSourceTerm} reduces to the classical \textit{inviscid Burgers equation} which has been studied in most textbooks on conservation laws \cite{Raviart,Lax,Serre,Smoller}. It is well known that the solution of the inviscid Burgers equation develops  discontinuities in finite time provided that the slope of the initial condition is negative at some point. 

For the homogeneous case $\mu=0$, system \eqref{DropModSimp} can be seen as the \textit{zero-pressure gas dynamics system} \cite{Bouchut2} or  as the \textit{sticky particle system} \cite{Sticky,weinan1996} that arises in the modeling of particles hitting and sticking to each other to explain the formation of large scale structures in the universe. The system of zero-pressure gas dynamics has been studied by several authors \cite{Bouchut2,Bouchut3,LBoudin,Sticky,weinan1996,Huang2001,LIYANG,Zhang,Sheng}. In particular, the existence of measure solutions for the Riemann problem was first presented by Bouchut \cite{Bouchut2}. Under suitable generalized Rankine-Hugoniot relation and entropy condition, the Riemann problem is constructively solved in \cite{Sheng}. 

If $\mu>0$ then  system \eqref{DropModSimp} is known as the \textit{Eulerian droplet model} \cite{Drop2,Bourgault2,Bourgault1,Bourgault4,Bourgault3,Drop3,Drop1}. This model is successfully used for the prediction of droplets impingement on airfoils and airplane wings during in-flight icing events \cite{Bourgault1,Bourgault3}. Extension to particle flows in airways was more recently attempted \cite{Bourgault2,Bourgault4}. The Eulerian droplet model has been studied  by several authors  at the numerical level \cite{Bourgault1} and at the practical level \cite{Drop2,Bourgault2,Bourgault4,Bourgault3,Drop3,Drop1}. To our knowledge, there is no theoretical study related to the system of zero-pressure gas dynamics including explicitly a right-hand side term as in \eqref{DropModSimp}. In this paper, we are interested in the theoretical study of the Eulerian droplet model \eqref{DropModSimp}. In reality, the drag coefficient $\mu$ is function of the droplet Reynolds number (see \cite{Bourgault1}). For performing analysis, we assume in the following that the drag coefficient $\mu$ and the carrier fluid  velocity $u_a$ are constant.

The Eulerian droplet model \eqref{DropModSimp} is a first-order system of conservation laws for the volume fraction $\alpha$ and the momentum $\alpha u$. For smooth solutions, it is equivalent to \eqref{DropModSimpAlphaU} which can be written in quasilinear form as
\begin{equation}
\begin{pmatrix}
\alpha\\
u
\end{pmatrix}_t
+
\begin{pmatrix}
u & \alpha\\
0 & u
\end{pmatrix}
\begin{pmatrix}
\alpha\\
u
\end{pmatrix}_x
=
\begin{pmatrix}
0\\
\mu(u_a-u)
\end{pmatrix}.
\end{equation}
The Jacobian matrix has one double eigenvalue $u$ and is not diagonalizable. Hence, system \eqref{DropModSimp} is \textit{weakly hyperbolic}. Systems of conservation laws in which hyperbolicity fails (because of eigenvalue coincidence) can encounter many difficulties, particularly in terms of boundedness of their solutions. To illustrate this recurrent difficulty with boundedness, consider the following linear first-order system  
\begin{eqnarray}
\left\lbrace
\begin{aligned}
&\begin{pmatrix}
\alpha\\
u
\end{pmatrix}_t
+
\begin{pmatrix}
\lambda & \beta \\
0 & \lambda
\end{pmatrix}
\begin{pmatrix}
\alpha\\
u
\end{pmatrix}_x
=0,\quad (x,t)\in \mathbb{R}\times\mathbb{R}^+,\\
&(\alpha,u)(x,0)=(\alpha_0,u_0)(x),\quad \forall x\in \mathbb{R},
\end{aligned}
\right.
\label{SHDIllus}
\end{eqnarray}
where $\lambda,\beta\neq 0$ are constant. System \eqref{SHDIllus} is weakly hyperbolic with one double eigenvalue $\lambda$. One can first solve the second equation of \eqref{SHDIllus} by the method of characteristics to find
\begin{equation}
u(x,t)=u_0(x-\lambda t)
\label{S1}
\end{equation} 
and then, considering $-\beta\partial_xu$ as a source term, we calculate the solution of the first equation
\begin{equation}
\alpha(x,t)=\alpha_0(x-\lambda t) - \beta tu_0^{\prime}(x-\lambda t).
\label{S2}
\end{equation} 
We immediately see that $\alpha$ is not defined in the classical sense at points where the initial condition $u_0$ is not differentiable. For instance, if $u_0$ is a Heaviside function then the expression for $\alpha$ would contain a Dirac $\delta$-function. The Cauchy problem, for bounded measurable data, is not of classical type. The concept of singular solutions incorporating Dirac $\delta$-functions along shock trajectories was first introduced in \cite{Korchinski}. Tan and Zhang \cite{TAN19941} studied a new type of waves,  delta-shock waves, as solution of nonlinear hyperbolic systems for which hyperbolicity fails. They proved that delta-shock waves are limit solutions to some reasonable viscous perturbations as the viscosity vanishes. A delta-shock wave is a generalization of an ordinary shock wave. Speaking informally, it is a kind of discontinuity, on which at least one of the state variables may develop an extreme concentration in the form of a Dirac $\delta$-function with the discontinuity in an other variable as its support. From the physical point of view, a delta-shock wave represents the process of concentration of mass. For related results on delta-shock waves, we refer to  \cite{Vacuum,CHEN2004,HCheng,KEYFITZ1995420,LIYANG,Sheng,YANG2014,YANG20125951,YIN2009} and the references therein.

The general purpose of this work is to solve the Riemann problem for the Eulerian droplet model \eqref{DropModSimp}. The arrangement of this paper is as follows. In section \ref{PerteReg}, we derive the condition for loss of regularity for a smooth solution of \eqref{DropModSimp}. In sections \ref{RPBurgersSourceTerm}, \ref{SectRPDropModAlphaU} and \ref{SectRiemProbDropModSimp}, we solve the Riemann problem for the inviscid Burgers equation with source term \eqref{BurgersSourceTerm}, system \eqref{DropModSimpAlphaU} and the Eulerian droplet model \eqref{DropModSimp}, respectively. In section \ref{GRHExisSol}, we investigate the existence of a solution to the generalized Rankine-Hugoniot conditions for \eqref{DropModSimp}. Test cases illustrating theoretical results are presented in section \ref{DropModSimpNumRes}.

\section{Loss of regularity for a smooth solution of the Eulerian droplet model}
\label{PerteReg} 
This section is devoted to the loss of regularity for smooth solutions of \eqref{DropModSimp}. For more details on loss of regularity for smooth solutions, we refer the readers to the work \cite{Serge,YANG1999447} on blowup of nonlinear hyperbolic systems, and to Whitham's classic text \cite{Whitham}. By a method similar to that in \cite{YANG1999447}, we prove that $\alpha$ and $\partial_xu$ blow up simultaneously in finite time even starting from smooth initial data.
 
Let $(\alpha,u)$ be a smooth solution of \eqref{DropModSimp} satisfying the initial condition
\begin{equation}
(\alpha,u)(x,0)=(\alpha_0,u_0)(x), \quad \alpha_0, u_0 \in \mathcal{C}^1(\mathbb{R}).
\label{SmoothInitCond}
\end{equation}
The characteristic curves $\chi=\chi(x,t;s)$ associated to \eqref{DropModSimp} are solutions of 
\begin{eqnarray}
\left\lbrace
  \begin{aligned}
    &\dfrac{d\chi}{ds}(x,t;s)=u\big(\chi(x,t;s),s\big),\quad s\in[0,T],\\
    &\chi(x,t;t)=x.
   \end{aligned}
\right.
\label{Caracteristique}
\end{eqnarray}
System \eqref{DropModSimp} reduces along these characteristics to
\begin{eqnarray}
\left\lbrace
  \begin{aligned}
    &\dfrac{D\alpha}{dt}=\alpha\partial_xu,\\
    &\dfrac{Du}{dt}=\mu(u_a-u),
   \end{aligned}
\right.
\label{DropModOnCharac}
\end{eqnarray}
where $\frac{D}{dt}=\frac{\partial}{\partial t}+u\frac{\partial}{\partial x}$ is the total derivative.
An integration of the second equation of \eqref{DropModOnCharac} gives
\begin{equation}
u\big(\chi(x,t;t),t\big)=u(x,t)=u_a + \big(u_0(x_0)-u_a\big)e^{-\mu t},
\label{sol}
\end{equation}
where $x_0=\chi(x,t;0)$. Substituting (\ref{sol}) in (\ref{Caracteristique}) and integrating, we  get
\begin{equation}
x=\chi(x,t;t)=x_0+u_at+\dfrac{(u_0(x_0)-u_a)(1-e^{-\mu t})}{\mu}.
\label{Xfunction}
\end{equation}
Hence, $x=x(x_0,t)$ can be seen as a function of $x_0$ and $t$, and thus $\partial_xu$ can be written as 
\begin{equation}
\partial_xu=\Big(\dfrac{\partial u}{\partial x_0}\dfrac{\partial x_0}{\partial x}+\dfrac{\partial u}{\partial t}\dfrac{\partial t}{\partial x}\Big).
\label{VolFracONCharac22}
\end{equation} 
As long as the characteristics do not intersect, the map
\begin{align*}
h: \mathbb{R}\times[0,\infty)&\rightarrow \mathbb{R}\times[0,\infty)\\
(x_0,t)\quad&\mapsto\quad(x,t)
\end{align*}
is bijective. The Jacobian matrix of $h$ and its inverse $h^{-1}$ are given by
\begin{equation}
J_h(x_0,t)=
\begin{pmatrix}
\frac{\partial x}{\partial x_0} & \frac{\partial x}{\partial t}  \\
0 & 1  
\end{pmatrix}
\quad \text{and} \quad 
J_{h^{-1}}(x,t)=
\begin{pmatrix}
\frac{\partial x_0}{\partial x} & \frac{\partial x_0}{\partial t} \\
\frac{\partial t}{\partial x} & 1
\end{pmatrix},
\end{equation}
respectively.  Since $J_{h^{-1}}(x,t)=J_h^{-1}(x_0,t)$  then
\begin{equation}
\frac{\partial x_0}{\partial x}=\frac{1}{\frac{\partial x}{\partial x_0}} \quad\text{and}\quad \frac{\partial t}{\partial x}=0.
\end{equation}
Hence, \eqref{VolFracONCharac22} reduces to
\begin{equation}
\partial_xu=\dfrac{\partial u}{\partial x_0}\dfrac{\partial x_0}{\partial x}=-\alpha\dfrac{\partial u}{\partial x_0}\dfrac{1}{\frac{\partial x}{\partial x_0}}.
\label{VolFracONCharac1}
\end{equation} 
Using \eqref{sol} and \eqref{Xfunction} in \eqref{VolFracONCharac1}, we obtain
\begin{equation}
\partial_xu=\dfrac{\mu e^{-\mu t}u_0'(x_0)}{\mu+(1-e^{-\mu t})u_0'(x_0)}.
\label{GradientU}
\end{equation}
The first equation of \eqref{DropModOnCharac} can be written now as 
\begin{equation}
\dfrac{D\alpha}{dt}=-\dfrac{\mu\alpha e^{-\mu t}u_0'(x_0)}{\mu+(1-e^{-\mu t})u_0'(x_0)}.
\label{VolFracONCharac}
\end{equation}
Assuming $\alpha\neq0$, one can  divide by $\alpha$ and integrate \eqref{VolFracONCharac} on both sides to obtain
\begin{equation}
\log\big(\alpha(x,t)\big)=-\log\big(K_D +(1-e^{-K_Dt})u_0'(x_0)\big)+\log(K_D)+\log\big(\alpha_0(x_0)\big).
\end{equation}
This last equality leads to
\begin{equation}
\alpha(x,t)=\dfrac{\mu\alpha_0(x_0)}{\mu +(1-e^{-\mu t})u_0'(x_0)}.
\label{SolVolFRacOnCharac}
\end{equation}
The following result holds:
\begin{prop}
\label{LossInitReg}
Let $(\alpha,u)$ be a smooth solution of \eqref{DropModSimp} and \eqref{SmoothInitCond}. Then $\alpha$ and $\partial_xu$ blow up if and only if there exists $x_0$ in the domain such that 
\begin{equation}
u_0'(x_0)<-\mu.
\label{NesSufCondLosReg}
\end{equation}
Moreover, the blowup occurs at 
\begin{equation}
t=\inf_{u_0'(x_0)<-\mu} \Big\lbrace-\dfrac{\log(1+\frac{\mu}{u_0'(x_0)})}{\mu}\Big\rbrace.
\label{LossRegTime}
\end{equation} 
\end{prop}
\begin{proof}
As $\alpha_0,u_0\in \mathcal{C}^1(\mathbb{R})$ then $\alpha$ and $\partial_xu$ blow up if and only if $\mu +(1-e^{-\mu t})u_0'(x_0)=0$. This happens if and only if 
\begin{equation}
u_0'(x_0)<0\quad\text{and}\quad 1-e^{-\mu t}=\dfrac{-\mu}{u_0'(x_0)}\Longleftrightarrow t=-\dfrac{\log(1+\frac{\mu}{u_0'(x_0)})}{\mu}.
\label{LossRegTimeProof}
\end{equation}
Since $1-e^{-\mu t}<1,\forall t\geq0$ then $u_0'(x_0)<-\mu$. The smallest time $t$ satisfying \eqref{LossRegTimeProof} is given by \eqref{LossRegTime}.
\end{proof}
\noindent
\begin{rem}
Inequality \eqref{NesSufCondLosReg} is also a necessary and sufficient condition for the characteristics to overlap. In fact, two characteristics $\chi_1(x,t;s)$ and $\chi_2(x,t;s)$ with distinct foots $x_1$ and $x_2$, respectively, cross each other if and only if there is $s^{*}>0$ such that $\chi_1(x,t;s^{*})=\chi_2(x,t;s^{*})$. By using \eqref{Xfunction} and the inequalities $0<1-e^{-\mu t}<1,\forall t>0$, the equality $\chi_1(x,t;s^{*})=\chi_2(x,t;s^{*})$ gives rise to
\begin{equation}
\dfrac{u_0(x_2)-u_0(x_1)}{x_2-x_1}<-\mu,
\label{CharacMeeting111}
\end{equation} 
 and by the mean value theorem, there exists a point $x_0$ such that
\begin{equation}
u_0'(x_0)=\dfrac{u_0(x_2)-u_0(x_1)}{x_2-x_1}<-\mu.
\label{CharacMeeting1}
\end{equation} 
\end{rem}

Thus, a smooth solution to \eqref{DropModSimp} loses its regularity if and only if the slope of the initial condition for $u$ is sufficiently negative with respect to the coefficient $\mu$.  This loss of regularity is reflected in a blowup of $\alpha$ and $\partial_xu$. This blowup leads to unboundedness and discontinuities in the solution. Therefore, no solution exists in the space of functions with bounded variation. We will investigate the form of the solution to the Riemann problem for \eqref{DropModSimp} by going through the solution of the Riemann problem for equation \eqref{BurgersSourceTerm} and then system \eqref{DropModSimpAlphaU}, respectively.

\section{Riemann problem for the inviscid Burgers equation with source term}
\label{RPBurgersSourceTerm}
In this section we study the inviscid Burgers equation with a zeroth order source term \eqref{BurgersSourceTerm} satisfying the initial condition
\begin{equation}
u(x,0)=u_0(x),
\label{InitCond}
\end{equation}
where  $u_0$ is a piecewise smooth function. The solution of the Riemann problem to the Burgers equation without a source term is either a rarefaction or a shock wave \cite{Lax,Smoller}. The solution to the Riemann problem for the Burgers equation with a discontinuous source term is constructed in \cite{FANG2012307}. It turns out that the discontinuity of the source term has clear influences on the shock or rarefaction waves generated by the initial Riemann data. In \cite{ZhangShen}, the shock wave solution for the inviscid Burgers equation with a linear forcing term is obtained by combining the Rankine-Hugoniot jump condition together with the method of characteristics, which reflects the impact of the inhomogeneous forcing term on the shock front. In these references, the term source is function of $x$ and $t$. In this section we solve the Riemann problem for the Burgers equation with a source term that depends on the solution $u$, using the method of characteristics. In addition, we refer to the  work \cite{ChangChouHongLin,mascia1997,SinestrariC} on how to use the method of characteristics to solve the Riemann problem for scalar conservation law with source term.

Due to the breaking of waves and formation of shocks, the initial value problem for \eqref{BurgersSourceTerm} does not generally possess globally defined smooth solutions, even when the initial data are very smooth. We showed in the previous section that  discontinuity appears in the solution of \eqref{BurgersSourceTerm}  if condition \eqref{NesSufCondLosReg} is satisfied.  Here, we look for the solution of the  Riemann problem for \eqref{BurgersSourceTerm}, i.e.\ the solution of \eqref{BurgersSourceTerm} and \eqref{InitCond}, where 
\begin{equation}
u_0(x)=
\left\lbrace
\begin{aligned}
&u_-,\quad x<0,\\
&u_+,\quad x>0,
\end{aligned}
\right.
\label{BurgersRPInitCond}
\qquad u_-,u_+ \in \mathbb{R}.
\end{equation}
We are particularly interested in the solution of the Riemann problem \eqref{BurgersSourceTerm} and \eqref{BurgersRPInitCond} because it will be useful in the resolution of the Riemann problem for system \eqref{DropModSimpAlphaU} in the next section. 

The solution of \eqref{BurgersSourceTerm} along the characteristics \eqref{Caracteristique} is given by \eqref{sol}. Substituting  \eqref{BurgersRPInitCond} in \eqref{Xfunction}, one obtains
\begin{eqnarray}
\chi(x,t;s)=
\left\lbrace
\begin{aligned}
&x_0 +u_as+\dfrac{u_a-u_-}{\mu}\big(e^{-\mu s}-1\big),\quad x_0<0,\\
&x_0 +u_as+\dfrac{u_a-u_+}{\mu}\big(e^{-\mu s}-1\big),\quad x_0>0.
\end{aligned}
\right.
\label{BurgersCharac}
\end{eqnarray}
\subsection{Shock waves}
We first assume that $u_->u_+$. In this case, condition \eqref{NesSufCondLosReg} is satisfied, and thus characteristics intersect within finite time. Some characteristic curves for different values of $u_a$ are represented in Figure~\ref{ShockFigCar}. 
\begin{figure}[!h]
\centering
\begin{tabular}{ccc}
\includegraphics[scale=0.31]{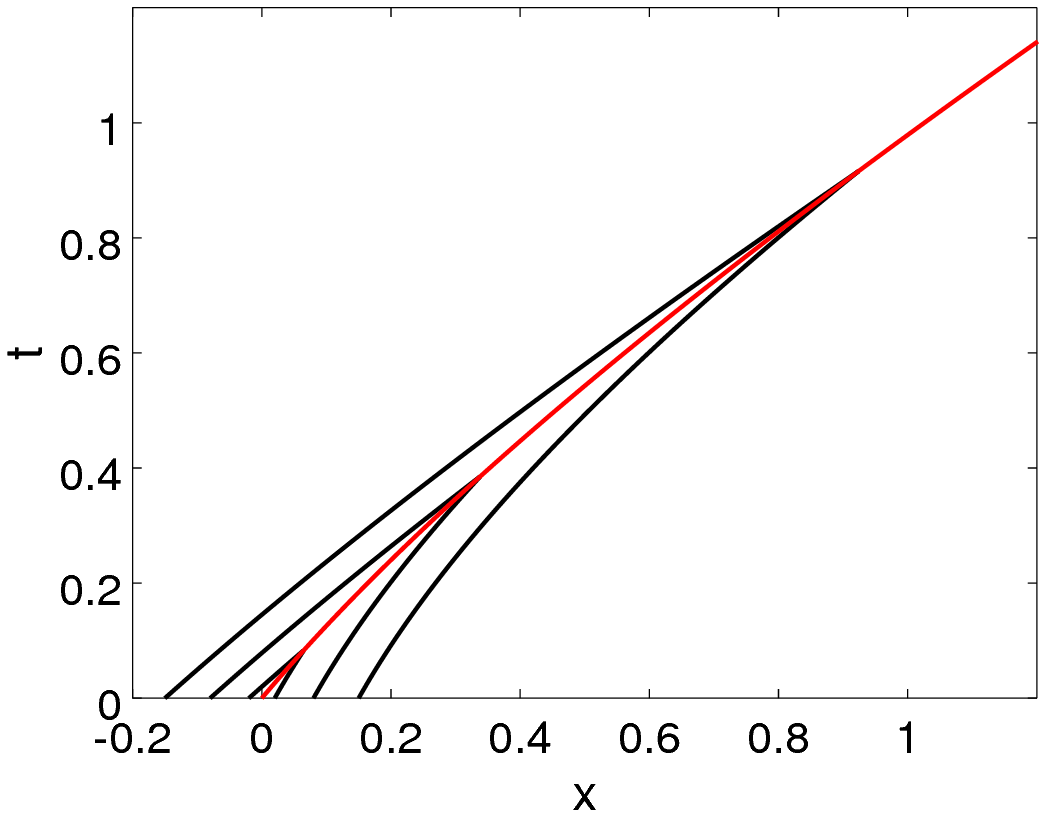} &\includegraphics[scale=0.31]{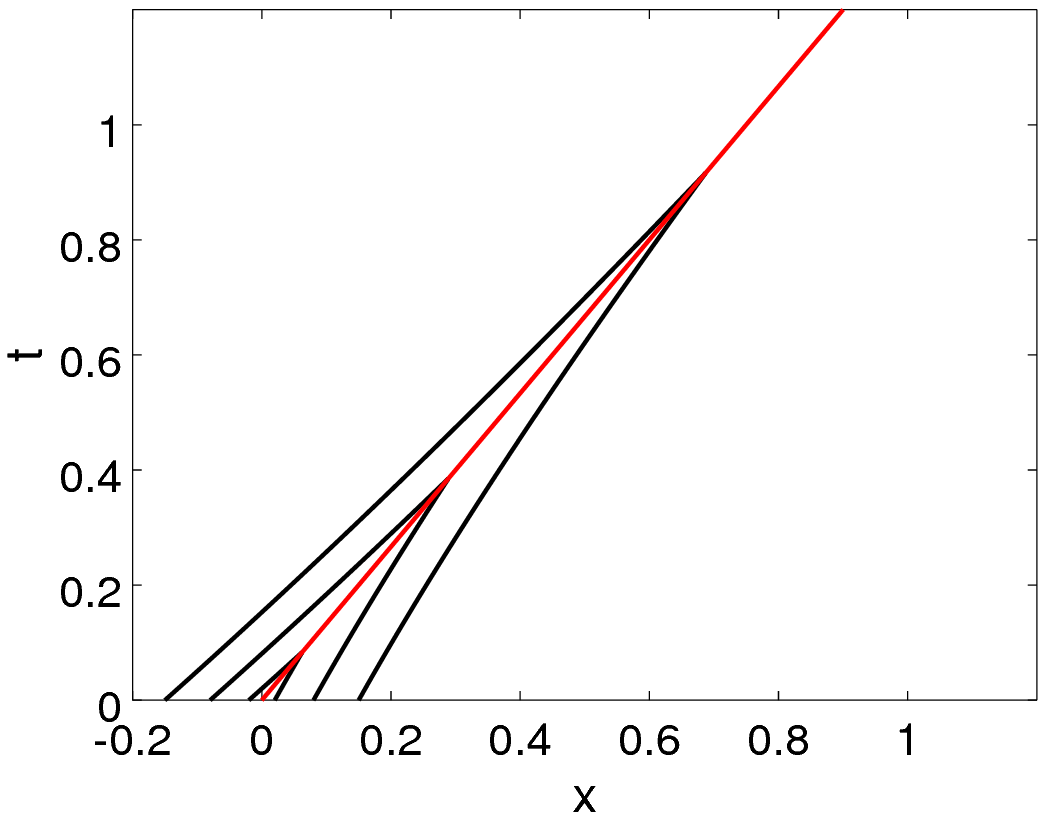}&\includegraphics[scale=0.31]{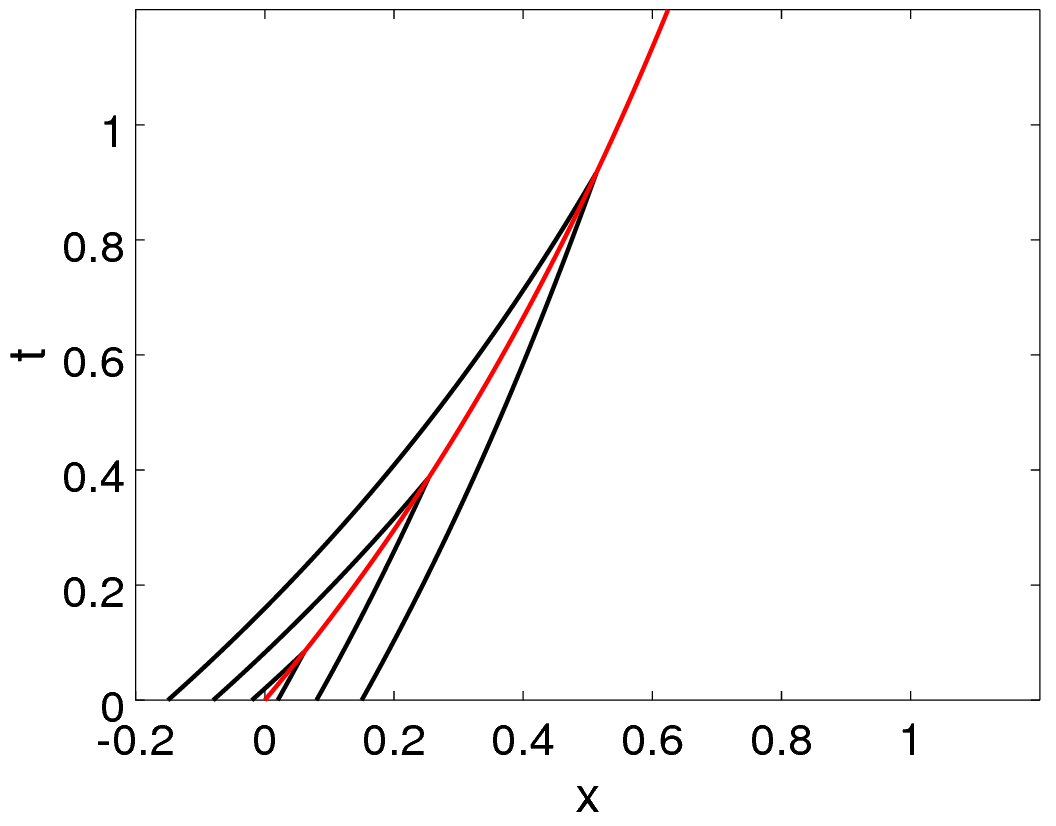}\\
\end{tabular}
\caption{Characteristic curves on the $x$-$t$ plane for $u_{-}=1.0$, $u_{+}=0.5$ and $\mu=1.0$. Left:  $u_a=1.5$, middle: $u_a=0.75$ and right: $u_a=0.2$.}
\label{ShockFigCar}
\end{figure}
The solution is a shock wave, i.e.\ a smooth curve $\Gamma=\lbrace (x,t) :\text{ } x=\xi(t), t\geq0\rbrace$ in the $x$-$t$ plane moving at speed $\sigma(t)=\xi'(t)$ and separating a left and right states denoted by $u_l(x,t)$ and $u_r(x,t)$, respectively. Solutions that may be discontinuous are taken in the weak sense. We have the following definition: 
\begin{defn}
\label{WeakDef}
We say that $u$ is a \textit{weak solution} of \eqref{BurgersSourceTerm} and \eqref{InitCond} if 
\begin{align}
\int_0^{\infty}\int_{-\infty}^{\infty}\big(u\psi_t+\frac{u^2}{2}\psi_x+\mu(u_a-u)\psi\big)dxdt=-\int_{-\infty}^{\infty}u_0(x)\psi(x,0)\,dx,
\label{BurgersWeakSolDef}
\end{align}
for all test functions $\psi$ $\in$ $\mathcal{C}_0^{\infty}(\overline{\mathbb{R}\times\mathbb{R^+}})$. 
\end{defn} 
\noindent
Let $u$ be a regular function  on both sides of the curve $\Gamma$, while being  discontinuous across this curve. We have the following characterization for a weak solution of \eqref{BurgersSourceTerm} and \eqref{InitCond}.
\begin{theo}
\label{BurgersMainTheoStatement}
The function $u$  is a weak solution of \eqref{BurgersSourceTerm} and \eqref{InitCond} if and only if  the following properties hold:
\begin{itemize}
\item[i)] $u$ satisfies \eqref{BurgersSourceTerm} in the classical sense on both sides of the curve $\Gamma$;
\item[ii)] $u(x,0)=u_0(x)$ for all $x\in\mathbb{R}$;
\item[iii)] the following Rankine-Hugoniot conditions are satisfied:
\begin{equation}
\big(u_r(t)-u_l(t)\big)\sigma(t)=\frac{1}{2}\big(u_r(t)^2-u_l(t)^2\big),
\label{BurgersRHC}
\end{equation} 
where $u_l(t)$ and $u_r(t)$ are the limit of the solution $u$ when $(x,t)$ approaches $(\xi(t),t)$ from the left and the right, respectively.
\end{itemize}
\end{theo}
\begin{proof}
The  proof  is  performed  as  for  the  classical  Burgers  equation (see \cite{Raviart,Serre})  but with  an  appropriate  treatment  of  the  source  term  which  disappears  in  the  Rankine-Hugoniot conditions.
\end{proof}
\begin{defn}
A discontinuity propagating with speed $\sigma$ given by \eqref{BurgersRHC} satisfies the entropy condition if 
\begin{equation}
u_r(t)<\sigma(t)<u_l(t).
\label{BurgersSourceTermEntropCond}
\end{equation}
\end{defn}
\noindent
Inequality  \eqref{BurgersSourceTermEntropCond} is known as \textit{Lax's entropy condition} \cite{Lax,Serre,Smoller}. It  means that all characteristics on both sides of the discontinuity are in-coming. This additional condition ensures the uniqueness of the Riemann solution  to the Burgers equation  without source term \cite{Evans2010}. The Lax's entropy condition is also used in \cite{FANG2012307,ZhangShen} for the Burgers equation with source term.

Returning to the Riemann problem for \eqref{BurgersSourceTerm}, Theorem \ref{BurgersMainTheoStatement} states that the solution $u$ satisfies \eqref{BurgersSourceTerm} in the classical sense on both sides of the curve $\Gamma$. The left and right states are determined from \eqref{sol}, that is
\begin{equation}
u_l(x,t)=u_a+(u_{-}-u_a)e^{-\mu t},\qquad u_r(x,t)=u_a+(u_{+}-u_a)e^{-\mu t}.
\label{UleftUright}
\end{equation}
These states are independent of $x$ away from the discontinuity, hence the limit states  $u_l(t)=u_l(x,t)$ and $u_r(t)=u_r(x,t)$. The shock speed of the shock wave 
\begin{equation}
\sigma(t)=\dfrac{1}{2}\big(u_l(t)+u_r(t)\big)=u_a+\left(\dfrac{u_-+u_+}{2}-u_a\right)e^{-\mu t},
\label{ShockSpeed}
\end{equation}
from \eqref{BurgersRHC} satisfies the entropy condition \eqref{BurgersSourceTermEntropCond}. The trajectory of the shock is given by
\begin{equation}
\begin{aligned}
\xi(t)&=\int_0^t\sigma(s)ds=u_at+\left(\dfrac{u_-+u_+-2u_a}{2\mu}\right)\big(1-e^{-\mu t}\big).
\end{aligned}
\label{DisconCurve}
\end{equation}
We reach the following result:
\begin{cor}
\label{BurgersTheoShockSol}
If $u_->u_+$ then the solution of the Riemann problem \eqref{BurgersSourceTerm} and \eqref{BurgersRPInitCond} is given by
\begin{eqnarray}
u(x,t)=
\left\lbrace
\begin{aligned}
&u_l(x,t),\qquad x<\xi(t),\\
&\sigma(t),\qquad\quad x=\xi(t),\\
&u_r(x,t),\quad\text{  } x>\xi(t),
\end{aligned}
\right.
\label{BurgersSourceTermShockSol}
\end{eqnarray}
where $u_l$, $u_r$ are given in \eqref{UleftUright}, $\sigma$ and $\xi$  are given in \eqref{ShockSpeed} and \eqref{DisconCurve}, respectively.
\end{cor}
\begin{rem}
\label{Rem1}
 We have:
\begin{equation}
\lim_{t\rightarrow\infty}\big(u_l(t)-\sigma(t)\big)=\lim_{t\rightarrow\infty}\big(\sigma(t)-u_r(t)\big)=0.
\end{equation}  
This means that the Lax's entropy condition for \eqref{BurgersSourceTerm} degenerates as time goes to infinity. This degeneracy is not observed with the classical inviscid Burgers equation ($\mu=0$) because the two limits states are constant.
\end{rem}

\subsection{Rarefaction waves}
\label{BurgersRarSect}
Secondly, we assume that $u_-<u_+$. Condition \eqref{NesSufCondLosReg} is not satisfied, and thus characteristics do not intersect, but do not cover the whole $x$-$t$ plane. Some characteristic curves for different values of $u_a$ are represented in Figure~\ref{RarFigCar}.
\begin{figure}[!h]
\centering
\begin{tabular}{ccc}
\includegraphics[scale=0.3]{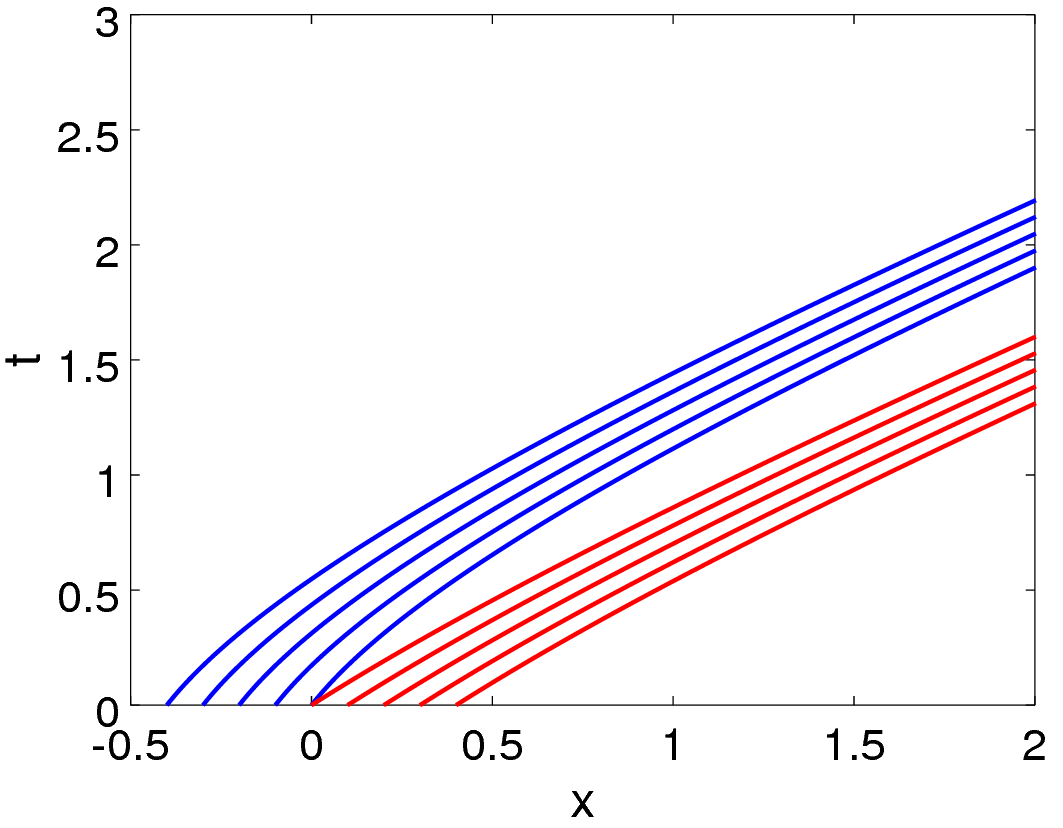} &\includegraphics[scale=0.293]{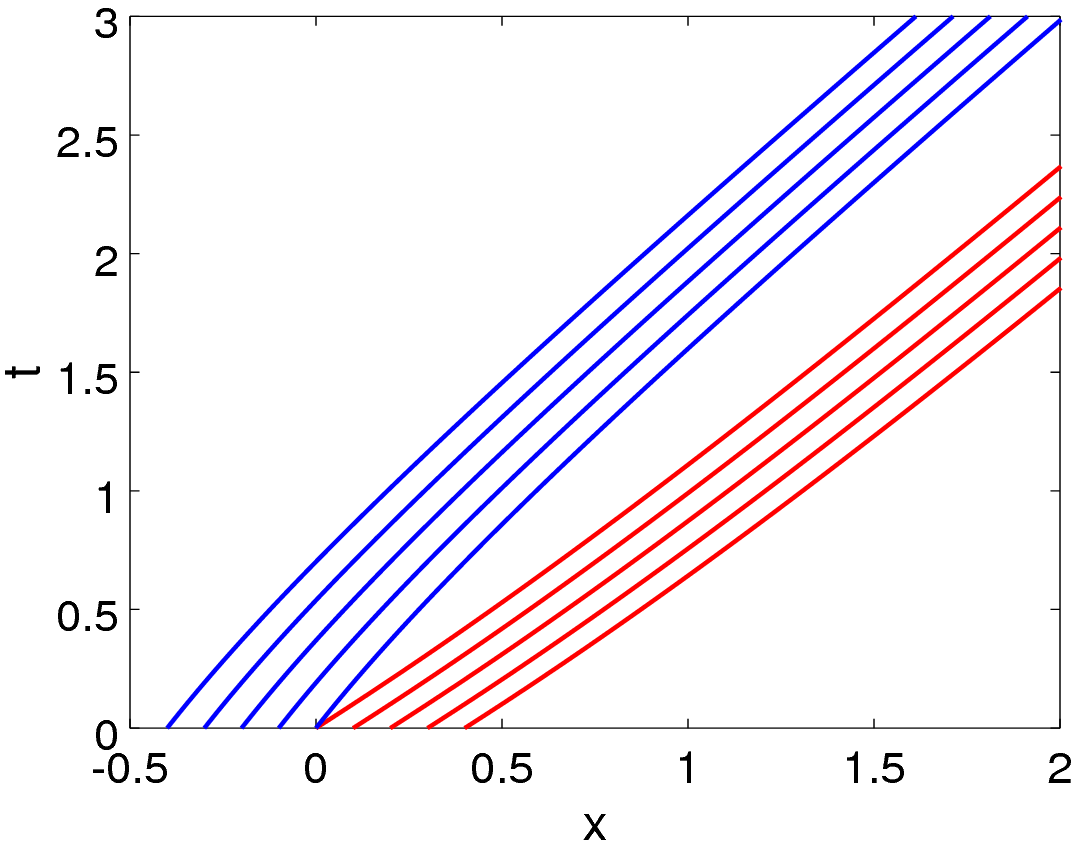}&\includegraphics[scale=0.3]{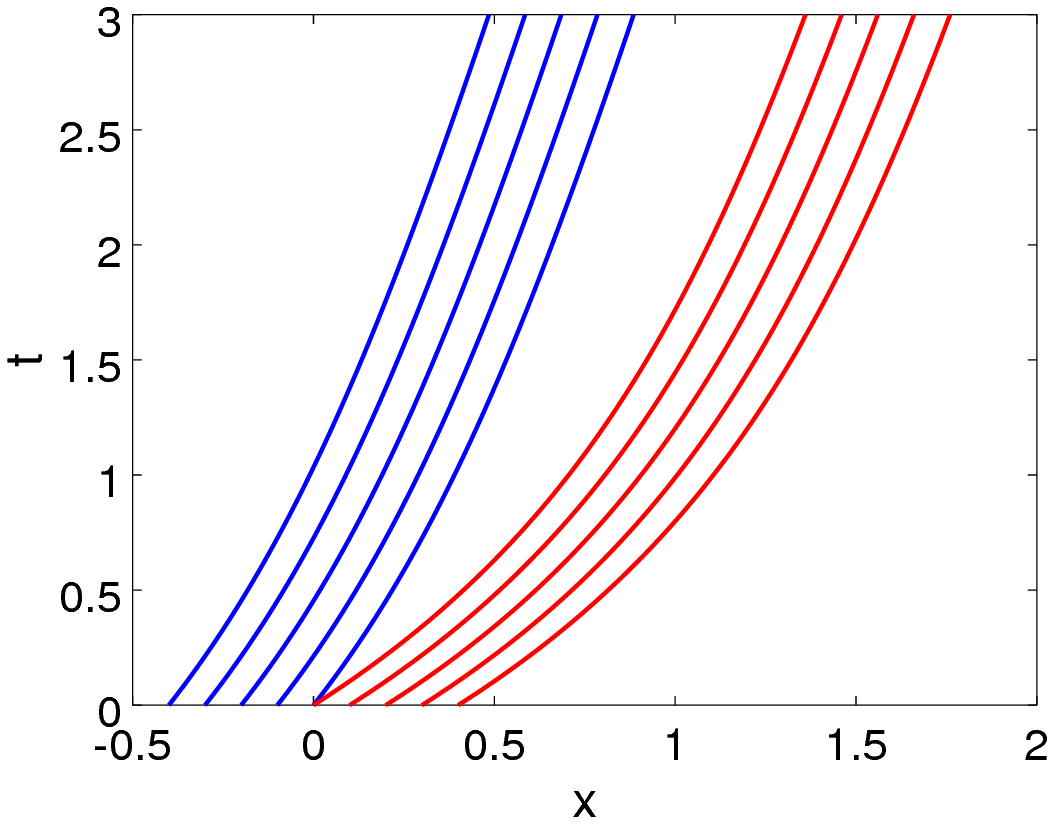}\\
\end{tabular}
\caption{Characteristic curves on the $x$-$t$ plane for $u_-=0.5$, $u_+=1.0$ and $\mu=1.0$. Left: $u_a=1.5$, middle: $u_a=0.75$ and right: $u_a=0.2$.}
\label{RarFigCar}
\end{figure}
The uncovered region $\mathcal{S}$ is delimited by the curves
\begin{equation} 
X_1(t)=\int_0^tu_l(s)ds=u_at+\dfrac{(u_a-u_-)(e^{-\mu t}-1)}{\mu}
\label{X1}
\end{equation}
and
\begin{equation}
X_2(t)=\int_0^tu_r(s)ds=u_at+\dfrac{(u_a-u_+)(e^{-\mu t}-1)}{\mu}.
\label{X2}
\end{equation}
By the method of characteristics, the solution of the Riemann problem is a rarefaction wave, i.e.\ a continuous function satisfying \eqref{BurgersSourceTerm}. This solution is given by $u_l(x,t)$ for $x<X_1(t)$ and $u_r(x,t)$ for $x>X_2(t)$. To find the solution inside $\mathcal{S}$,  we have to fill this region by a family of characteristics starting at the origin, i.e.\ to solve \eqref{Caracteristique} with the initial condition $\chi(x,t;0)=0$. Figure~\ref{FilledSurfaceS} shows the region $\mathcal{S}$ (delimited by $x=X_1(t)$ (in blue) and $x=X_2(t)$ (in red)) filled with a family of characteristics (in black) starting at the origin. 
\begin{figure}[!h]
\centering
\begin{tabular}{c}
\includegraphics[scale=0.3]{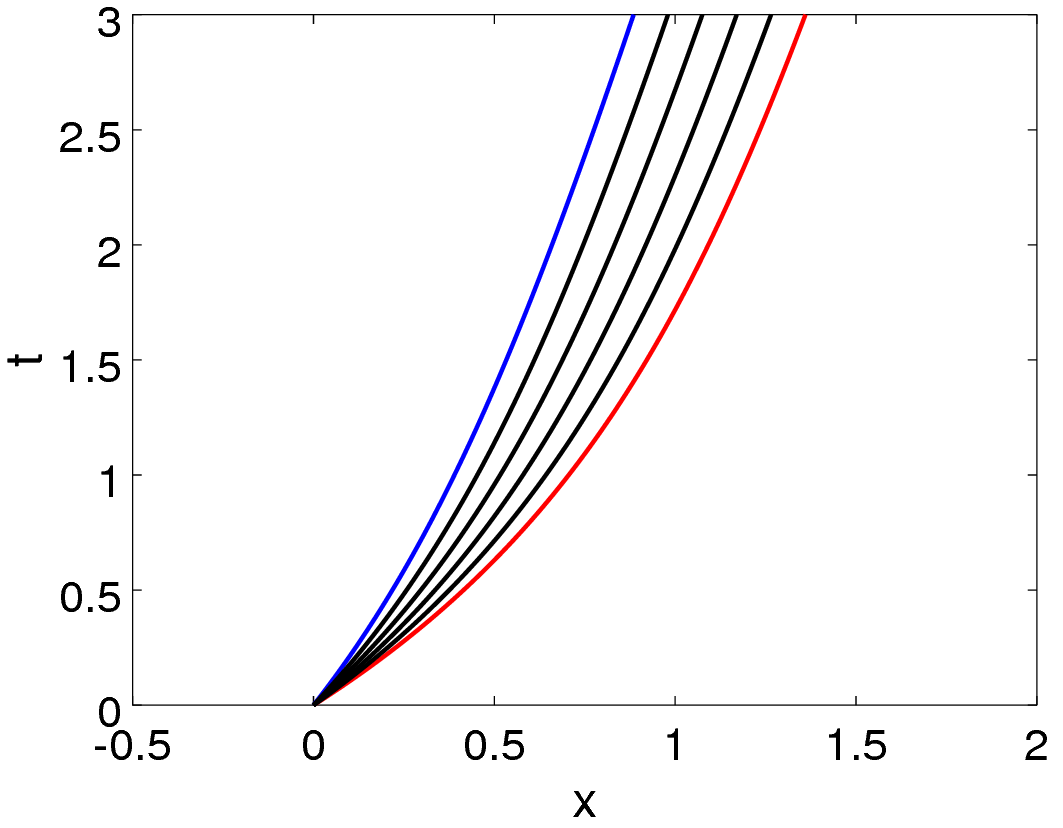}
\end{tabular}
\caption{Region $\mathcal{S}$ filled with characteristics starting at the origin. $u_-=0.5$, $u_+=1.0$, $\mu=1.0$ and $u_a=0.2$.}
\label{FilledSurfaceS}
\end{figure}
From  \eqref{sol}, we  calculate the solution $u$ at the foot of the characteristics. We obtain
\begin{equation}
u(\chi(x,t;0),0)=u(x_0,0)=u_a+\big(u(x,t)-u_a\big)e^{\mu t}.
\label{SolInit}
\end{equation} 
Substituting  \eqref{SolInit} in  \eqref{Xfunction}, one gets
\begin{eqnarray}
\begin{aligned}
\chi(x,t;s)=u_as-\dfrac{\big(u(x,t)-u_a\big)\big(e^{\mu(t-s)}-e^{\mu t}\big)}{\mu}.
\end{aligned}
\end{eqnarray}
From this last equation, we get  for $s=t$ that
\begin{equation}
u(x,t)=\overline{u}(x,t)=u_a+\dfrac{\mu\big(x-u_at\big)}{e^{\mu t}-1}.
\label{RarSolInter}
\end{equation}
The following result holds: 
\begin{cor}
\label{BurgersRarSol}
If $u_-<u_+$ then the solution of the Riemann problem \eqref{BurgersSourceTerm} and \eqref{BurgersRPInitCond} is given by
\begin{equation}
u(x,t)=
\left\lbrace
\begin{aligned}
&u_l(x,t),\qquad x<X_1(t),\\
&\overline{u}(x,t),\qquad X_1(t)\leq x \leq X_2(t),\\
&u_r(x,t),\qquad x>X_2(t),
\end{aligned}
\right.
\label{BurgersSourceTermRarSol}
\end{equation}
where $u_l$, $u_r$ are given in \eqref{UleftUright}, $X_1$, $X_2$ are given in \eqref{X1} and \eqref{X2}, respectively, and $\overline{u}$ is given in \eqref{RarSolInter}.
\end{cor}
\begin{proof}
The function $u$ satisfies \eqref{BurgersSourceTerm} inside and outside $\mathcal{S}$, and is continuous at points $X_1(t)$ and $X_2(t)$ for all $t>0$. In fact, by replacing $X_1(t)$ (resp.\ $X_2(t)$) in  \eqref{RarSolInter}, one gets $u_l$ (resp.\  $u_r$).
\end{proof}

The Burgers equation with source term develops discontinuities if and only if the slope of the initial condition is sufficiently negative with respect to the coefficient $\mu$. This is an extension of the condition for loss of regularity for the classical Burgers equation (for $\mu=0$). The characteristic curves associated to the inviscid Burger equation with source term are no longer straight lines but are curves that tend asymptotically to straight lines as time grows, as opposed to the classical Burgers equation. The solution of the Riemann problem is either a shock or a rarefaction wave as in the homogeneous case. However, the left and right states are no longer constant, while asymptotically approaching $u_a$ which behaves as an equilibrium point as time goes to infinity. The zeroth order linear source term acts as a relaxation force preventing shocks from occurring and the Lax's entropy condition degenerates as time goes to infinity.

\section{Riemann problem for system \eqref{DropModSimpAlphaU}}
\label{SectRPDropModAlphaU}
In this section we study system \eqref{DropModSimpAlphaU} satisfying the initial conditions
\begin{equation}
\big(\alpha,u\big)(x,0)=\big(\alpha_0,u_0\big)(x),
\label{DropModInAlphaUInitCond}
\end{equation}
where $\alpha_0$ and $u_0$ are piecewise smooth functions. For $\mu=0$, system \eqref{DropModSimpAlphaU} is used to model the evolution of density inhomogeneities in matter in the universe (see \cite{Shandarin}, section II.B.3). Although viscosity is mentioned in this reference, it is set to zero before solving.  For a complete solution of system \eqref{DropModSimpAlphaU} in the homogeneous case $\mu=0$, we refer the reader to \cite{weinan1996}. Here, we solve the Riemann problem for \eqref{DropModSimpAlphaU} with the  initial condition
\begin{equation}
\big(\alpha_0,u_0\big)(x)=
\left\lbrace
\begin{aligned}
(\alpha_-,u_-),\quad x<0,\\
(\alpha_+,u_+),\quad x>0,
\end{aligned}
\right.
\qquad \alpha_-, \alpha_+  \in \mathbb{R}^+ \text{ and } u_-, u_+ \in \mathbb{R}.
\label{DropModInAlphaURPInitCond}
\end{equation}
The solution of \eqref{DropModSimpAlphaU} and \eqref{DropModInAlphaURPInitCond} will be useful in the resolution of the Riemann problem for  \eqref{DropModSimp} in the next section.
\subsection{Delta-shock waves}
Assume $u_{-}>u_{+}$. The characteristics overlap. As pointed out in section \ref{PerteReg}, the solution is not bounded, more precisely, $\alpha$ blows up and $u$ is discontinuous. This leads to the fundamental question of defining products of non-smooth solutions.  A suitable notion of weak solutions for nonconservative systems involving product of non-smooth functions was proposed by Dal Maso, LeFloch, and Murat \cite{DalMaso} and the nonlinear stability of such solutions was investigated therein. We are interested here in the conservative form and solutions are sought in the sense of distributions. Motivated by  \cite{HCheng,CHENG201117,Sheng,TAN19941,YANG1999447}, we seek solutions with $\delta$-distribution at the jump, i.e.\ we look for a solution in the form 
\begin{equation}
\alpha(x,t)=\alpha^0(x,t)+\omega(t)\delta(x-\xi(t)),\quad u(x,t)=u^0(x,t),
\label{ShockFunctionsAlphaU}
\end{equation}
where  $\alpha^0$, $u^0$ are smooth functions on both sides of the curve 
\begin{equation}
\Gamma=\lbrace(x,t):x=\xi(t), t\geq0\rbrace,
\end{equation} 
while being discontinuous across this curve, $\delta=\delta(x)$ is the Dirac mass centered at the origin and $\omega$ is a smooth function defined on $\mathbb{R}_0^+$ and satisfying the initial condition
\begin{equation}
\omega(0)=\omega_0 \in \mathbb{R}_0^+.
\label{weightedFunctionInitCond}
\end{equation}
We define a weighted $\delta$-function $\omega(t)\delta_{\xi}$ supported on the curve $\Gamma$ as 
\begin{equation}
\langle\omega(t)\delta_{\xi},\psi\rangle=\int_0^{\infty}\omega(t)\psi(\xi(t),t)dt,
\label{WeightedDeltaFunction}
\end{equation}
for all test function $\psi$ $\in$ $\mathcal{C}_0^{\infty}(\mathbb{R}\times\mathbb{R^+})$. We also define the  duality products between the functions $\alpha$ and $u$ (seen as distributions) and test functions in $\mathcal{C}_0^{\infty}(\mathbb{R}\times\mathbb{R^+})$ as 
\begin{align}
\langle \alpha,\psi\rangle&=\int_0^{\infty}\int_{-\infty}^{\infty}\alpha^0\psi\,dxdt +\langle \omega(t)\delta_{\xi},\psi\rangle,\\
\langle \alpha u,\psi\rangle&= \int_0^{\infty}\int_{-\infty}^{\infty}\alpha^0u^0\psi\,dxdt +\langle \omega(t)\xi'(t)\delta_{\xi},\psi\rangle,
\label{DefDisSolModEE}
\end{align}
and we introduce the following definition: 
\begin{defn}
\label{Def}
We say that a pair of distributions $(\alpha,u)$ as given in \eqref{ShockFunctionsAlphaU}-\eqref{WeightedDeltaFunction} is a \textbf{weak solution} of \eqref{DropModSimpAlphaU} and \eqref{DropModInAlphaUInitCond} if 
\begin{align}
\label{WeakSolModEE}
&\langle \alpha, \psi_t \rangle + \langle \alpha u, \psi_x \rangle =-\int_{-\infty}^{\infty}\alpha_0(x)\psi(x,0)\,dx-\omega_0\psi(\xi(0),0),\\
&\int_0^{\infty}\int_{-\infty}^{\infty}\Big(u\psi_t+\frac{u^2}{2}\psi_x+\mu(u_a-u)\psi\Big)dxdt=-\int_{-\infty}^{\infty}u_0(x)\psi(x,0)\,dx,
\label{WeakSolModEEE}
\end{align}
hold for all test functions $\psi \in \mathcal{C}_0^{\infty}(\overline{\mathbb{R}\times\mathbb{R^+}})$.
\end{defn}
\noindent
We have the following result:
\begin{theo}
\label{DropModSimpAlphaUMainTheo}
A pair of distributions $(\alpha,u)$ as given in \eqref{ShockFunctionsAlphaU}-\eqref{WeightedDeltaFunction} is a weak solution of \eqref{DropModSimpAlphaU} and \eqref{DropModInAlphaUInitCond} if and only if the following properties are satisfied:
\begin{itemize}
\item[i)]$(\alpha^0,u^0)$ satisfies \eqref{DropModSimpAlphaU} in the classical sense on both sides of the  curve $\Gamma$;
\item[ii)] $\alpha^0(x,0)=\alpha_0(x)$ and $u^0(x,0)=u_0(x)$ for all $x\in \mathbb{R}$;
\item[iii)] the following system of differential-algebraic equations (DAE) is satisfied on $\Gamma$:
\begin{equation}
\left\lbrace
\begin{aligned}
&\dfrac{d\omega}{dt}(t)=
\big(\alpha_r(t)-\alpha_l(t)\big)\sigma(t)-\big(\alpha_r(t)u_r(t)-\alpha_l(t)u_l(t)\big),\\
&\big(u_r(t)-u_l(t)\big)\sigma(t)=\dfrac{1}{2}\big(u_r(t)^2-u_l(t)^2\big),\\
\end{aligned}
\right.
\label{GenRHCDropModSimpAlphaU}
\end{equation}
where $(\alpha_l(t),u_l(t))$ and $(\alpha_r(t),u_r(t))$ are the limit of the solution $(\alpha,u)$ when $(x,t)$ approaches $(\xi(t),t)$ from the left and the right, respectively;
\item[iv)] the following initial condition is satisfied:
\begin{equation}
\omega(0)=\omega_0.
\label{WeiFuncInitCond}
\end{equation}
\end{itemize}
\end{theo}
\begin{proof}
By Theorem \ref{BurgersMainTheoStatement}, $u=u^0$ is a weak solution of the second equation of \eqref{DropModSimpAlphaU} satisfying the second equality in the initial conditions \eqref{DropModInAlphaUInitCond} if and only if the properties $(i)$-$(iii)$ are satisfied by $u$.
It remains to prove the properties $(i)$-$(iv)$ for $\alpha$.
Let $\psi \in \mathcal{C}_0^{\infty}(\overline{\mathbb{R}\times\mathbb{R^+}})$. We have: 
\begin{align*}
\langle \alpha, \psi_t \rangle + \langle \alpha u, \psi_x \rangle
%&=\int_0^{\infty}\int_{-\infty}^{\infty}\alpha^0\psi_t\,dxdt+\langle \omega(t)\delta_{\xi},\psi_t\rangle 
%+\int_0^{\infty}\int_{-\infty}^{\infty}\alpha^0 u^0\psi_x\,dxdt+\langle \omega(t)\xi'(t)\delta_{\xi},\psi_x\rangle \nonumber\\
&=\int_0^{\infty}\int_{-\infty}^{\infty}\big(\alpha^0\psi_t+\alpha^0 u^0\psi_x\big)\,dxdt +\int_0^{\infty}\big(\omega(t)\psi_t(\xi(t),t)+
\omega(t)\xi'(t)\psi_x(\xi(t),t)\big)\,dt\nonumber\\
&=\int_0^{\infty}\int_{-\infty}^{\infty}(\alpha^0\psi_t + \alpha^0 u^0\psi_x)\,dxdt
+\int_0^{\infty}\omega(t)\frac{d\psi}{dt}(\xi(t),t)\,dt.
\end{align*}
Integrating by part the r.h.s integrals, we obtain 
\begin{align}
&\langle \alpha, \psi_t \rangle + \langle \alpha u, \psi_x \rangle =-\int_0^{\infty}\int_{-\infty}^{\xi(t)}\big(\alpha^0_t+(\alpha^0u^0)_x\big)\psi\,dxdt-\int_0^{\infty}\int_{\xi(t)}^{\infty}\big(\alpha^0_t+(\alpha^0u^0)_x\big)\psi\,dxdt\nonumber\\
&\quad+\int_{0}^{\infty}\Big(\big(\alpha_r(t)-\alpha_l(t)\big)\sigma(t)-\big(\alpha_r(t)u_r(t)-\alpha_l(t)u_l(t)\big)-\frac{d\omega}{dt}(t)\Big)\psi(\xi(t),t)\,dt\nonumber\\
&\qquad-\int_{-\infty}^{\infty}\alpha^0(x,0)\psi(x,0)\,dx-\omega(0)\psi(\xi(0),0).
\label{DropModSimpAlphaUMainTheoProof}
\end{align}
\textbf{1)} Suppose that $(\alpha,u)$ is a weak solution of \eqref{DropModSimpAlphaU} and \eqref{DropModInAlphaUInitCond}. Then \eqref{DropModSimpAlphaUMainTheoProof} reduces to
\begin{align}
&-\int_{-\infty}^{\infty}\alpha_0(x)\psi(x,0)\,dx-\omega_0\psi(\xi(0),0)=-\int_0^{\infty}\int_{-\infty}^{\xi(t)}\big(\alpha^0_t+(\alpha^0u^0)_x\big)\psi\,dxdt\nonumber\\
&\quad-\int_0^{\infty}\int_{\xi(t)}^{\infty}\big(\alpha^0_t+(\alpha^0u^0)_x\big)\psi\,dxdt-\int_{-\infty}^{\infty}\alpha^0(x,0)\psi(x,0)\,dx-\omega(0)\psi(\xi(0),0)\nonumber\\
&\quad+\int_{x=\xi(t)}\Big(\big(\alpha_r(t)-\alpha_l(t)\big)\sigma(t)-\big(\alpha_r(t)u_r(t)-\alpha_l(t)u_l(t)\big)-\frac{d\omega}{dt}(t)\Big)\psi(\xi(t),t)\,dt.
\label{FirstImpl}
\end{align}
\textit{\underline{Case 1}}: Taking $\psi$ $\in$ $\mathcal{C}_0^{\infty}(\mathbb{R}\times\mathbb{R^+})$ satisfying $\psi(\xi(t),t)=0$ for all $t\geq0$ in \eqref{FirstImpl}, we get $(i)$ .\\
\textit{\underline{Case 2}}: Taking $\psi$ $\in$ $\mathcal{C}_0^{\infty}(\overline{\mathbb{R}\times\mathbb{R^+}})$ satisfying $\psi(\xi(t),t)=0$ for all $t\geq0$ and $\psi(x,0)\neq0$ in \eqref{FirstImpl} and using $(i)$, we obtain $(ii)$.\\
\textit{\underline{Case 3}}: Taking $\psi$ $\in$ $\mathcal{C}_0^{\infty}(\mathbb{R}\times\mathbb{R^+})$ in \eqref{FirstImpl} and using $(i)$, one gets $(iii)$.\\
\textit{\underline{Case 4}}: Taking $\psi$ $\in$ $\mathcal{C}_0^{\infty}(\overline{\mathbb{R}\times\mathbb{R^+}})$ in  \eqref{FirstImpl} and using $(i)$, $(ii)$ and $(iii)$ then we obtain $(iv)$.\\ 
\textbf{2)} Conversely, if $(i)$, $(ii)$, $(iii)$ and $(iv)$ are satisfied then \eqref{DropModSimpAlphaUMainTheoProof} reduces to \eqref{WeakSolModEE}.
\end{proof}
\begin{rem}
The DAE \eqref{GenRHCDropModSimpAlphaU}-\eqref{WeiFuncInitCond} reflect the exact relationship between the limit states on the two sides, the weight and propagation speed of the discontinuity, as the classical Rankine-Hugoniot conditions do for ordinary shocks. It is called the \textbf{generalized Rankine-Hugoniot (GRH) conditions}. The equations \eqref{GenRHCDropModSimpAlphaU}  state that the defect of mass conservation induced by the discontinuous  velocity at the shock leads to a mass accumulation along the trajectory of the shock.
\end{rem}

\begin{defn}
We call a \textbf{$\delta$-shock solution} or \textbf{delta-shock wave} of system\eqref{DropModSimpAlphaU} and \eqref{DropModInAlphaUInitCond}, a weak solution of \eqref{DropModSimpAlphaU} and \eqref{DropModInAlphaUInitCond} satisfying the entropy condition \eqref{BurgersSourceTermEntropCond}.
\end{defn}

The above definition of delta-shock waves was used by Sheng and Zhang \cite{Sheng} to construct the solution to the Riemann problem for the zero-pressure gas dynamics system. The classical Rankine-Hugoniot conditions for ordinary shock waves have been generalized to those of delta-shock waves to describe the relationship among  the limit states, propagation speed, location and weight of the discontinuity \cite{Sheng}.    Delta-shock waves can be interpreted as particular measure solutions as defined in \cite{Cheng1997}, and belong to the space of signed Borel measures on $\mathbb{R}$, denoted  $\mathcal{M}(\mathbb{R})$. A measure-theoretic justification of delta-shock waves can also be found in \cite{Bouchut2,weinan1996,li2003,TAN19941}. Since the volume fraction $\alpha$ may develop a Dirac $\delta$-function in finite time, it is natural to seek solutions for \eqref{DropModSimpAlphaU} in the sense of measures, i.e.\ in the sense of distributions which are signed measures. The Cauchy problem for the system of zero-pressure gas dynamics  is constructed in \cite{Cheng1997} with measure solutions in the space $\mathcal{M}(\mathbb{R})$.  For initial conditions taken in the space of bounded measurable functions, the uniqueness of solutions to the zero-pressure gas dynamics equations is established in \cite{WANG1997341} with the Oleinik entropy condition. Huang and Wang \cite{Huang2001} established the uniqueness of the weak solution to zero-pressure gas dynamics equations when the initial condition is a Random measure. They showed that, besides the Oleinik entropy condition, it is also important to require the energy to be weakly continuous initially. This condition is called the energy condition. For initial data in the space of measures, it was proven in \cite{li2003} that the Oleinik condition is not sufficient to ensure uniqueness for measure solutions, but it has to be complemented with a cohesion condition.  The Lax's entropy condition \eqref{BurgersSourceTermEntropCond} is used in \cite{Vacuum,CHEN2004,Sheng} as a first step towards the uniqueness of delta-shock solutions to the system of zero-pressure gas dynamics.

In the following, we construct a delta-shock solution to \eqref{DropModSimpAlphaU} that satisfies the Lax's entropy condition \eqref{BurgersSourceTermEntropCond}. Theorem \ref{DropModSimpAlphaUMainTheo} states that the solution $(\alpha,u)$ is smooth on both side of the curve $\Gamma$. From \eqref{SolVolFRacOnCharac}, we get $\alpha(x,t)=\alpha_0(x_0)$ on both sides of $\Gamma$ since $u_0'=0$ on both sides of this curve. Hence, the left and right states for the solution $\alpha$ are determined by the initial data, that is
\begin{equation}
\alpha_l(x,t)=\alpha_l(t)=\alpha_- \quad \text{and} \quad \alpha_r(x,t)=\alpha_r(t)=\alpha_+.
\label{Alphaleftright}
\end{equation}
By Corollary \ref{BurgersTheoShockSol}, $u$ is given by \eqref{BurgersSourceTermShockSol}, and therefore the limit states  are given by \eqref{UleftUright}.  Substituting $\eqref{UleftUright}$, $\eqref{ShockSpeed}$ and $\eqref{Alphaleftright}$ in the first equation in \eqref{GenRHCDropModSimpAlphaU} and integrating the latter, we obtain the weighted function
\begin{equation}
\omega(t)=\omega_0+ \dfrac{(\alpha_++\alpha_-)(u_--u_+)}{2\mu}\big(1-e^{-\mu t}\big).
\label{Weightfunction}
\end{equation}
For the  Riemann problem, we take $\omega_0=0$. In general, one can start with a $\delta$-shock solution as an initial condition and $\omega_0$ is not necessarily zero. We can now state the following result:
\begin{cor}
If $u_->u_+$ then  the Riemann problem \eqref{DropModSimpAlphaU} and \eqref{DropModInAlphaURPInitCond} has a $\delta$-shock solution
given by
 \begin{equation}
(\alpha,u)(x,t)=
 \left\lbrace
  \begin{aligned}
   &\big(\alpha_-,u_l(x,t)\big),\qquad x<\xi(t),\\
   &\big(\omega(t)\delta(x-\xi(t)),\sigma(t)\big),\quad x=\xi(t),\\
   &\big(\alpha_+,u_r(x,t)\big),\qquad x>\xi(t),
  \end{aligned}
 \right.
 \label{DropModAlphaUShockSol}
\end{equation}
where  $u_l$, $u_r$ are given in \eqref{UleftUright}, $\omega$ is given in \eqref{Weightfunction} with $\omega_0=0$,  $\sigma$  and $\xi$ are given by \eqref{ShockSpeed} and \eqref{DisconCurve}, respectively.
\end{cor}

\subsection{Two contact discontinuities with a vacuum state}
\label{SubsectTwoContDis}
Assume $u_-<u_+$. By Corollary \ref{BurgersRarSol}, $u$ is given by \eqref{BurgersSourceTermRarSol} and is independent of $x$ outside the region $\mathcal{S}$, hence 
\begin{equation}
\dfrac{D\alpha}{dt}=-\alpha\partial_x u=0, \quad \forall (x,t) \notin \mathcal{S}.
\label{AlphaOnCarac77}
\end{equation}
Outside $\mathcal{S}$, $\alpha(x,t)=\alpha_0(x_0)$ and is determined by the Riemann initial data, that is  
\begin{equation}
\alpha(x,t)=
\left\lbrace
\begin{aligned}
&\alpha_-, \quad x<X_1(t),\\
&\alpha_+, \quad x>X_2(t).
\end{aligned}
\right.
\end{equation}
Inside $\mathcal{S}$, $u$ is given by \eqref{RarSolInter}. We have:
\begin{equation}
\dfrac{D\alpha}{dt}=-\alpha\partial_x\overline{u}, \quad \forall (x,t) \in \mathcal{\overline{S}}.
\label{AlphaOnCarac7}
\end{equation}
The function $\alpha=0$ satisfies \eqref{AlphaOnCarac7}. Let $\epsilon>0$ and assume $\alpha\neq0$. One divides \eqref{AlphaOnCarac7} by $\alpha$ and integrates on both sides from $s=\epsilon$ to $s=t$ to obtain 
\begin{equation}
\alpha(x,t)=\dfrac{\big(1-e^{-\mu\epsilon}\big)e^{\mu t}\alpha(x,\epsilon)}{e^{\mu t}-1},\quad \forall (x,t) \in \mathcal{\overline{S}}.
\label{AlphaRarInter}
\end{equation}
Taking $\epsilon\rightarrow 0$ in \eqref{AlphaRarInter}, one obtains
\begin{equation}
\alpha(x,t)=0,\quad \forall (x,t) \in \mathcal{\overline{S}}.
\end{equation}
We reach the following result:
\begin{cor}
If $u_-<u_+$ then the solution of the Riemann problem \eqref{DropModSimpAlphaU} and \eqref{DropModInAlphaURPInitCond} is given by
\begin{eqnarray}
\big(\alpha,u\big)(x,t)=
 \left\lbrace
  \begin{aligned}
   &\big(\alpha_-,u_l(x,t)\big),\quad x<X_1(t),\\
   &\big(0,\overline{u}(x,t)\big),\quad X_1(t)\leq x \leq X_2(t),\\
   &\big(\alpha_+,u_r(x,t)\big),\quad x>X_2(t),\\
  \end{aligned}
 \right.
\label{VolFracRarSol}
\end{eqnarray}   
where $u_l$ and $u_r$ are given in \eqref{UleftUright}, $\overline{u}$ is given in \eqref{RarSolInter}, $X_1$ and $X_2$ are given in \eqref{X1} and \eqref{X2}, respectively.
\end{cor}
\noindent
Note that $u$ is continuous while $\alpha$ might be discontinuous across the curves $x=X_1(t)$ and $x=X_2(t)$. This type of solution is called a \textit{two-contact-discontinuity}. The two-contact-discontinuity solution \eqref{VolFracRarSol} contains a vacuum state (region where $\alpha=0$). Vacuum states are important physical states in fluid mechanics and often yield singularities in the physical systems, which cause essential analytical and numerical difficulties (see \cite{DiPerna1983,FANG2012307,Hoff,LIN,lions1996,LIU1980}).

\section{Riemann problem for the Eulerian droplet model}
\label{SectRiemProbDropModSimp}
In this section we construct a solution to the Riemann problem for \eqref{DropModSimp} using the results from the previous sections. 

\subsection{Delta-shock waves}
\label{CompDeltaShock}
Assume $u_->u_+$. The characteristics overlap.  As pointed out in section \ref{PerteReg}, a bounded solution of \eqref{DropModSimp} does not exist.  Motivated by the discussion and results in the previous section, we seek for solution in the form given in \eqref{ShockFunctionsAlphaU}-\eqref{WeightedDeltaFunction}.  We define the following duality products between $\alpha$ and $u$ (seen as distributions) and test functions in $\mathcal{C}_0^{\infty}(\mathbb{R}\times\mathbb{R^+})$:
\begin{align}
&\langle \alpha,\psi\rangle=\int_0^{\infty}\int_{-\infty}^{\infty}\alpha^0\psi dxdt +\langle \omega(t)\delta_{\xi},\psi\rangle,\\
&\langle \alpha u,\psi\rangle= \int_0^{\infty}\int_{-\infty}^{\infty}\alpha^0u^0\psi dxdt +\langle \omega(t)\xi'(t)\delta_{\xi},\psi\rangle,\\
&\langle \alpha u^2,\psi\rangle= \int_0^{\infty}\int_{-\infty}^{\infty}\alpha^0(u^0)^2\psi dxdt +\langle \sigma(t)\omega(t)\xi'(t)\delta_{\xi},\psi\rangle,\\
&\langle\alpha(u_a-u),\psi\rangle= \int_0^{\infty}\int_{-\infty}^{\infty}\alpha^0(u_a-u^0)\psi dxdt +\langle (u_a-\sigma(t))\omega(t)\delta_{\xi},\psi\rangle,
\label{DefDisSol}
\end{align}
where $\sigma(t)=\xi'(t)$ satisfies 
\begin{equation}
\sigma(0)=\sigma_0\in \mathbb{R}.
\end{equation}
\begin{defn}
\label{Def2}
We say that a pair of distributions $(\alpha,u)$ as given in \eqref{ShockFunctionsAlphaU}-\eqref{WeightedDeltaFunction} is a \textbf{weak solution} of \eqref{DropModSimp} and \eqref{DropModInAlphaUInitCond} if 
\begin{align}
\label{WeakSol1}
&\langle \alpha, \psi_t \rangle + \langle \alpha u, \psi_x \rangle=-\int_{-\infty}^{\infty}\alpha_0(x)\psi(x,0)\,dx-\omega_0\psi(\xi(0),0),\\
&\langle \alpha u, \psi_t \rangle + \langle \alpha u^2, \psi_x \rangle + \mu\langle \alpha(u_a-u),\psi\rangle =-\int_{-\infty}^{\infty}\alpha_0(x)u_0(x)\psi(x,0)\,dx-\sigma_0\omega_0\psi(\xi(0),0),
\label{WeakSol2}
\end{align}
hold for all test functions $\psi \in \mathcal{C}_0^{\infty}(\overline{\mathbb{R}\times\mathbb{R^+}})$.
\end{defn} 
\noindent
We have the following characterization for a weak solution of \eqref{DropModSimp} and \eqref{DropModInAlphaUInitCond}:
\begin{theo}
\label{TheoShockSol}
A pair of distributions $(\alpha,u)$ as given in \eqref{ShockFunctionsAlphaU}-\eqref{WeightedDeltaFunction} is a weak solution of \eqref{DropModSimp} and \eqref{DropModInAlphaUInitCond} if and only if  the following properties are satisfied:
\begin{itemize}
\item[i)] $(\alpha^0,u^0)$ satisfies \eqref{DropModSimp} in the classical sense on both sides of the curve $\Gamma$;
\item[ii)] $\alpha^0(x,0)=\alpha_0(x)$ and $\alpha^0(x,0)u^0(x,0)=\alpha_0(x)u_0(x)$ for all $x\in \mathbb{R}$;
\item[iii)] the following ODEs are satisfied on the curve $\Gamma$:
\begin{eqnarray}
\left\lbrace
\begin{aligned}
&\dfrac{d\omega}{dt}=\big(\alpha_r(t)-\alpha_l(t)\big)\sigma(t)-\big(\alpha_r(t)u_r(t)-\alpha_l(t)u_l(t)\big),\qquad\qquad\qquad\qquad\qquad\qquad\qquad\\
&\dfrac{d(\omega\sigma)}{dt}=\big(\alpha_r(t)u_r(t)-\alpha_l(t)u_l(t)\big)\sigma(t)-\big(\alpha_r(t)u_r(t)^2-\alpha_l(t)u_l(t)^2\big)
 +\mu\big(u_a-\sigma(t)\big)\omega(t),
\end{aligned}
\right.
\label{GenRHCDropModSimp}
\end{eqnarray}
where $(\alpha_l(t),u_l(t))$ and $(\alpha_r(t),u_r(t))$ are the limit of the solution $(\alpha,u)$ when $(x,t)$ approaches $(\xi(t),t)$ from the left and the right, respectively;
\item[iv)] the following initial conditions are satisfied:
\begin{equation}
\omega(0)=\omega_0,\quad \sigma(0)\omega(0)=\sigma_0\omega_0.
\label{GenRHCDropModSimpInitCond}
\end{equation}
\end{itemize}
\end{theo}
\begin{proof}
Proceed as in the proof of Theorem \ref{DropModSimpAlphaUMainTheo}.
\end{proof}

\begin{defn}
We call a \textbf{$\delta$-shock solution} or \textbf{delta-shock wave} of system \eqref{DropModSimp} and \eqref{DropModInAlphaUInitCond}, a weak solution of \eqref{DropModSimp} and \eqref{DropModInAlphaUInitCond} satisfying the Lax's entropy condition \eqref{BurgersSourceTermEntropCond}.
\end{defn}
\noindent
We now return to the solution of the Riemann problem for \eqref{DropModSimp}. By Theorem \ref{TheoShockSol}, the solution of \eqref{DropModSimp} is smooth  on both sides of the curve $x=\xi(t)$. Hence, it is given by the solution of \eqref{DropModSimpAlphaU} on both sides of this curve.  We have the following result:
\begin{cor}
If $u_->u_+$ then the  Riemann problem  \eqref{DropModSimp} and \eqref{DropModInAlphaURPInitCond} has a $\delta$-shock solution  given by 
\begin{equation}
(\alpha,u)(x,t)=
 \left\lbrace
  \begin{aligned}
   &\big(\alpha_-,u_l(x,t)\big),\qquad x<\xi(t),\\
   &\big(\omega(t)\delta(x-\xi(t)),\sigma(t)\big),\quad x=\xi(t),\\
   &\big(\alpha_+,u_r(x,t)\big),\qquad x>\xi(t),
  \end{aligned}
 \right.
\label{DropModSimpShockSol}
\end{equation}
where $u_l$, $u_r$ are given by\eqref{UleftUright}, $(\omega,\sigma)$ is the solution of the GRH conditions \eqref{GenRHCDropModSimp}-\eqref{GenRHCDropModSimpInitCond} with $\omega_0=0$, and  $\xi$ is given by
\begin{equation}
\xi(t)=\int_0^{t}\sigma(s)\,ds.
\label{AA}
\end{equation}
\end{cor}
\noindent
We will prove in the next section the existence of a solution to the GRH conditions \eqref{GenRHCDropModSimp}-\eqref{GenRHCDropModSimpInitCond} satisfying the Lax's entropy condition \eqref{BurgersSourceTermEntropCond}.

\subsection{Two contact discontinuities with a vacuum state}
Assume $u_-<u_+$. The Riemann solution is constructed in the same manner as in  subsection \ref{SubsectTwoContDis}. It is given by the two-contact-discontinuity with vacuum state \eqref{VolFracRarSol}. 

The solution of the Riemann problem for \eqref{DropModSimp} is either a delta-shock wave or a two-contact-discontinuity. The systems  \eqref{DropModSimp} and \eqref{DropModSimpAlphaU} are equivalent for smooth and two-contact-discontinuity solutions but they differ for $\delta$-shock solutions. In fact, one can check with a tedious calculation that the solution of the DAE \eqref{GenRHCDropModSimpAlphaU}-\eqref{WeiFuncInitCond} does not satisfy the GRH conditions \eqref{GenRHCDropModSimp}-\eqref{GenRHCDropModSimpInitCond}.  

\section{Existence of a solution to the GRH conditions satisfying the Lax's entropy}
\label{GRHExisSol}
In this section we prove the existence of a solution to \eqref{GenRHCDropModSimp}-\eqref{GenRHCDropModSimpInitCond} satisfying the Lax's entropy condition \eqref{BurgersSourceTermEntropCond}. The following results hold:
\begin{lem} (\textbf{Growth of the point mass $\omega$})\\
\label{Lem}
Assume  that $\alpha_l(t),\alpha_r(t)>0$ and $u_l(t)>u_r(t)$ for all $t\geq0$. Suppose that there exists a solution $(\omega,\sigma)\in\mathcal{C}^1(\mathbb{R}^+)\times\mathcal{C}^1(\mathbb{R}^+)$ to the initial value problem \eqref{GenRHCDropModSimp}-\eqref{GenRHCDropModSimpInitCond}. If $\sigma$ satisfies \eqref{BurgersSourceTermEntropCond}  then
\begin{equation}
\omega(t_2)>\omega(t_1),\quad \forall t_2>t_1.
\label{OmgIncr1}
\end{equation}
In particular, for the solution of the Riemann problem \eqref{DropModSimp} and \eqref{DropModInAlphaURPInitCond}, we have
\begin{equation}
\omega(t)\geq \omega_0+\dfrac{K(1-e^{-\mu t})}{\mu},\quad \forall t\geq0,
\label{OmgIncr}
\end{equation}
where $K=\min\lbrace\alpha_-,\alpha_+\rbrace(u_--u_+)$.
\end{lem}
\begin{proof}We proceed by cases. Let $t\geq0$.\\
\textit{\underline{Case 1}}: Assume that $\alpha_l(t)=\alpha_r(t)$. The first equation of \eqref{GenRHCDropModSimp} reduces to
\begin{align*}
%\dfrac{d\omega}{dt}(t)=\alpha_l(t)\big(u_l(t)-u_r(t)\big)=\alpha_l(t)(u_--u_+)e^{-\mu t},\quad \forall t\geqslant0.%\geq \alpha_-(u_--u_+)e^{-%\mu T},\quad \forall t\in[0,T].
\dfrac{d\omega}{dt}(t)=\alpha_l(t)\big(u_l(t)-u_r(t)\big)>0.
\end{align*}
\textit{\underline{Case 2}}: Assume that $\alpha_l(t)<\alpha_r(t)$. From the first inequality in \eqref{BurgersSourceTermEntropCond}, the first equation of \eqref{GenRHCDropModSimp} leads to
\begin{align*}
\dfrac{d\omega}{dt}(t)&=\big(\alpha_r(t)-\alpha_l(t)\big)\sigma(t)-\big(\alpha_r(t)u_r(t)-\alpha_l(t)u_l(t)\big)
>\alpha_l(t)\big(u_l(t)-u_r(t)\big)>0.
%&>\alpha_-\big(u_l(t)-u_r(t)\big)=\alpha_l(t)(u_--u_+)e^{-\mu t},\quad \forall t\geqslant0.%\geq \alpha_-(u_--u_+)e^{-\mu T},\quad \forall t\in[0,T].
\end{align*}
\textit{\underline{Case 3}}: Assume that $\alpha_l(t)>\alpha_r(t)$. From the second inequality in \eqref{BurgersSourceTermEntropCond}, the first equation of \eqref{GenRHCDropModSimp} leads to
\begin{align*}
\dfrac{d\omega}{dt}(t)&=\big(\alpha_r(t)-\alpha_l(t)\big)\sigma(t)-\big(\alpha_r(t)u_r(t)-\alpha_l(t)u_l(t)\big)>\alpha_r(t)\big(u_l(t)-u_r(t)\big)>0.
\end{align*}
Combining these three cases, we obtain
\begin{equation}
\dfrac{d\omega}{dt}(t)>0,\quad \forall t\geq0.
%\dfrac{d\omega}{dt}(t)\geq Ke^{-\mu T},\quad \forall t\geqslant0.
\label{Help1}
\end{equation}
Hence, \eqref{OmgIncr1} holds. 
In particular, for the solution of the Riemann problem \eqref{DropModSimp} and \eqref{DropModInAlphaURPInitCond}, the limit states are given by
$\alpha_l(t)=\alpha_-$, $\alpha_r(t)=\alpha_+$, $u_l(t)=u_a+(u_--u_a)e^{-\mu t}$ and $u_l(t)=u_a+(u_+-u_a)e^{-\mu t}$. By setting $K=\min\lbrace\alpha_-,\alpha_+\rbrace(u_--u_+)$, one obtains
\begin{equation}
\dfrac{d\omega}{dt}(t)\geq Ke^{-\mu t},\quad \forall t\geq0.
\label{Help1111}
\end{equation}
Integrating this last inequality  on both sides from $0$ to $t$, we get
\begin{equation}
\omega(t)\geq\omega(0)+K\int_0^te^{-\mu s}\,ds=\omega_0+\dfrac{K(1-e^{-\mu t})}{\mu},\quad\forall t\geq0.
\end{equation}
\end{proof}

\begin{prop} (\textbf{The shock speed satisfies the Lax's entropy condition})\\
\label{DeltaShockEntropCond}
Assume that $\alpha_l(t), \alpha_r(t)>0$, $u_l(t)>u_r(t)$ for all $t\geq0$ and $\sigma_0 \in (u_r(0),u_l(0))$. If $(\omega,\sigma)\in\mathcal{C}^1(\mathbb{R}^+)\times\mathcal{C}^1(\mathbb{R}^+)$ is a solution to the initial value problem \eqref{GenRHCDropModSimp}-\eqref{GenRHCDropModSimpInitCond} then $\sigma$ satisfies the Lax's entropy condition \eqref{BurgersSourceTermEntropCond} for all $t\geq0$.
\end{prop}
\begin{proof}
We proceed by contradiction. Suppose that there exists $t\geq0$ such that \eqref{BurgersSourceTermEntropCond} is not satisfied. From $\sigma(0)=\sigma_0\in(u_r(0),u_l(0))$ and the continuity of $\sigma$, there exists a smallest $t>0$, denoted $t^*$, such that
\begin{equation}
\sigma(t^*)=u_l(t^*) \text{ or } \sigma(t^*)=u_r(t^*) \text{ and } \sigma \text{ satisfies } \eqref{BurgersSourceTermEntropCond},\text{ } \forall t \in [0,t^*). 
\label{Absurde}
\end{equation} 
By Lemma \ref{Lem}, $\omega$ satisfies \eqref{OmgIncr1} on $[0,t^*)$. Combining this with the continuity of $\omega$, we get
\begin{equation}
\omega(t^*)>\omega(0)=\omega_0\geqslant0.
\label{WeightfunctionPos}
\end{equation}
Assume that  $\sigma(t^*)=u_l(t^*)$. Substituting $\sigma(t^*)$ by $u_l(t^*)$ in \eqref{GenRHCDropModSimp}, one can calculate
\begin{equation}
\begin{aligned}
\omega(t^*)\dfrac{d\sigma}{dt}(t^*)&=\dfrac{d(\omega\sigma)}{dt}(t^*)-\sigma(t^*)\dfrac{d\omega}{dt}(t^*)
%&=\dfrac{d(\omega\sigma)}{dt}(t^*)-u_l(t^*)\dfrac{d\omega}{dt}(t^*)\\
%&=\alpha_+u_r(t^*)\big(u_--u_+\big)e^{-\mu t^*}+\mu\big(u_a-u_-\big)e^{-\mu t^*}\omega(t^*) \nonumber\\
%&\qquad-\alpha_+u_l(t^*)\big(u_--u_+\big)e^{-\mu t^*} \nonumber\\
%&=\alpha_+\big(u_--u_+\big)\big(u_r(t^*)-u_l(t^*)\big)e^{-\mu t^*}
%+\mu\big(u_a-u_-\big)e^{-\mu t^*}\omega(t^*) \nonumber \\
=-\alpha_r(t^*)\big(u_l(t^*)-u_r(t^*)\big)^2+\mu\big(u_a-u_l(t^*)\big)\omega(t^*).
\end{aligned}
\label{OmgDifSig}
\end{equation}
Using this last equation, one gets
\begin{equation}
\begin{aligned}
\omega(t^*)\dfrac{d(\sigma-u_l)}{dt}(t^*)&=\omega(t^*)\dfrac{d\sigma}{dt}(t^*)-\omega(t^*)\dfrac{du_l}{dt}(t^*)
%&=-\alpha_r(t^*)\big(u_l(t^*)-u_r(t^*)\big)^2-\Big(\dfrac{du_l}{dt}(t^*)-\mu(u_a-u_l(t^*))\Big)\omega(t^*)\\
=-\alpha_r(t^*)\big(u_l(t^*)-u_r(t^*)\big)^2<0.
\end{aligned}
\label{Helps20}
\end{equation} 
Using \eqref{WeightfunctionPos}, we deduce from \eqref{Helps20} that $$\dfrac{d(\sigma-u_l)}{dt}(t^*)<0.$$ By the continuity of the function $t\mapsto \dfrac{d(\sigma-u_l)}{dt}(t)$, there exists $\epsilon>0$ such that $$\dfrac{d(\sigma-u_l)}{dt}(t)<0,\quad \forall t\in]t^*-\epsilon,t^*[.$$
Integrating this last inequality on both sides  from $t^*-\epsilon$ to $t^*$, we obtain
$$0\geq\int_{t^*-\epsilon}^{t^*}\dfrac{d(\sigma-u_l)}{ds}(s)ds=(\sigma-u_l)(t^*)-(\sigma-u_l)(t^*-\epsilon)=-\sigma(t^*-\epsilon)+u_l(t^*-\epsilon).$$
This last inequality implies that $\sigma(t^*-\epsilon)\geq u_l(t^*-\epsilon)$ which contradicts \eqref{Absurde}.\\
Assume that $\sigma(t^*)=u_r(t^*)$. A similar reasoning as above leads also to a contradiction of \eqref{Absurde}. 
Thus, the shock speed $\sigma$ satisfies \eqref{BurgersSourceTermEntropCond} for all $t\geq0$.
\end{proof}
\noindent
We next state the result for the existence of a solution to \eqref{GenRHCDropModSimp}-\eqref{GenRHCDropModSimpInitCond}  satisfying the \eqref{BurgersSourceTermEntropCond}. To simplify our notation, we set
\begin{eqnarray}
\begin{aligned}
&a(t)=\alpha_r(t)-\alpha_l(t),\quad b(t)=\alpha_r(t)u_r(t)-\alpha_l(t)u_l(t),\quad\\
&c(t)=\alpha_r(t)u_r(t)^2-\alpha_l(t)u_l(t)^2, \quad \theta(t)=\omega(t)\sigma(t).
\end{aligned}
\label{DefFunc}
\end{eqnarray}
The functions $a$, $b$ and $c$ are continuous and bounded for all $t\geqslant0$ since the limit states $(\alpha_l,u_l)$ and $(\alpha_r,u_r)$  are continuous and bounded. Substituting $a$, $b$, $c$ and $\theta$ in  \eqref{GenRHCDropModSimp}, this system  can be rewritten now as 
\begin{equation}
\left\lbrace
\begin{aligned}
&\dfrac{d\omega}{dt}(t)=a(t)\dfrac{\theta(t)}{\omega(t)}-b(t),\\
&\dfrac{d\theta}{dt}(t)=b(t)\dfrac{\theta(t)}{\omega(t)}+\mu\big(u_a\omega(t)-\theta(t)\big)-c(t).
\end{aligned}
\right.
\label{Mod2GenRHCDropModSimp}
\end{equation}
The initial value problem \eqref{GenRHCDropModSimp}-\eqref{GenRHCDropModSimpInitCond} can then be  rewritten in the following condensed form
\begin{equation}
\left\lbrace
\begin{aligned}
&\dfrac{d\mathbf{z}}{dt}=\mathbf{f}(\mathbf{z},t),\\
&\mathbf{z}(0)=(\omega_0,\theta_0)^T,\text{ with }\theta_0=\omega_0\sigma_0,
\end{aligned}
\right.
\label{CondGRH}
\end{equation}
where
\begin{equation}
\mathbf{z}(t)=
\begin{pmatrix}
\omega(t)\\
\theta(t)
\end{pmatrix}\quad
\text{and} \quad
\mathbf{f}(\mathbf{z},t)=
\begin{pmatrix}
f_1(\mathbf{z},t)\\
f_2(\mathbf{z},t)
\end{pmatrix}=
\begin{pmatrix}
a(t)\dfrac{\theta(t)}{\omega(t)}-b(t)\\
b(t)\dfrac{\theta(t)}{\omega(t)}+\mu\big(u_a\omega(t)-\theta(t)\big)-c(t)
\end{pmatrix}.
\label{CondGRHDetails}
\end{equation}
\begin{theo}(\textbf{Existence of a solution to the GRH conditions})\\
\label{GRHExistence}
If $\alpha_-,\alpha_+>0$, $u_->u_+$ and $\sigma_0 \in (u_+,u_-)$ then the GRH conditions \eqref{GenRHCDropModSimp}-\eqref{GenRHCDropModSimpInitCond} for the solution of the Riemann problem \eqref{DropModSimp} and \eqref{DropModInAlphaURPInitCond} have a $\mathcal{C}^1$-solution $(\omega,\sigma)$ satisfying \eqref{BurgersSourceTermEntropCond} for all $t\geqslant 0$.
\end{theo}
\begin{proof}
i) Suppose that $\omega_0>0$. Cauchy-P\'eano theorem (see \cite{crouzeix1984analyse}, p.~59) ensures the existence of a $\mathcal{C}^1$-solution $(\omega,\theta)$ to \eqref{CondGRH} on some interval $[0,t_1]$.  From Proposition \ref{DeltaShockEntropCond} and Lemma \ref{Lem}, $\omega(t)>0$ for all $t\in[0,t_1]$. This implies that the partial derivatives $\frac{\partial f_i}{\partial z_j}$ exist and are continuous on  $\mathbb{R}^+\times\mathbb{R}\times[0,t_1]$. Cauchy-Lipschitz existence theorem (see \cite{crouzeix1984analyse}, p.~65) ensures the uniqueness of this solution on the interval $[0,t_1]$. Define 
\begin{equation}
\sigma(t)=\frac{\theta(t)}{\omega(t)},\quad \forall t\in[0,t_1].
\end{equation}
Clearly, $(\omega,\sigma)$ is a $\mathcal{C}^1$-solution of \eqref{GenRHCDropModSimp}-\eqref{GenRHCDropModSimpInitCond} on $[0,t_1]$.  Using the  Lemmas 3.2 and 3.3 in \cite{crouzeix1984analyse}, Proposition \ref{DeltaShockEntropCond} and Lemma \ref{Lem}, this solution can be extended to all $t\geq0$. 

ii) Suppose that $\omega_0=0$. Take any finite $T>0$.  Let $F_1=\lbrace(\omega_n,\theta_n)\rbrace_{n\geq1}$ be the family of functions such that $(\omega_n,\theta_n)$ is a $\mathcal{C}^1$-solution of \eqref{CondGRH} on the interval $[0,T]$ satisfying the initial conditions $\omega_n(0)=\frac{1}{n}$ and $\theta_n(0)=\frac{\sigma_0}{n}$.  Moreover, the functions  $\omega_n$ satisfy \eqref{OmgIncr1} and \eqref{OmgIncr}. The existence of the family $F_1$ follows from part $(i)$.
Let us prove that $F_1$ is a family of bounded and equicontinuous functions on the interval $[0,T]$.  For all $n\geq1$, we have:
\begin{equation}
\left\lbrace
\begin{aligned}
&\dfrac{d\omega_{n}}{dt}(t)=a(t)\sigma_{n}(t)-b(t),\\
&\dfrac{d\theta_n}{dt}(t)=b(t)\sigma_{n}(t)+\mu\big(u_a-\sigma_{n}(t)\big)\omega_{n}(t)-c(t).
\end{aligned}
\right.
\label{GenRHCFamily}
\end{equation}
Taking the absolute value in the  first equation of \eqref{GenRHCFamily},  we get 
\begin{align}
\left|\dfrac{d\omega_{n}}{dt}(t)\right|\leqslant \left|a(t)\right|\left|u_l(t)\right|+ \left|b(t)\right|\leqslant M,
\label{Helps5}
\end{align}
where $M=\max_{s\in[0,T]}\big\lbrace \vert a(s)\vert\vert u_l(s)\vert+\vert b(s)\vert\big\rbrace.$ Integrating $\omega_{n}'$ from $0$ to $t$, taking the absolute value and using \eqref{Helps5}, we get
\begin{align*}
\left|\omega_n(t)\right|&=\left|\omega_n(0)+\int_0^t \omega_{n}'(s)\,ds\right|
\leqslant \frac{1}{n}+\int_0^t\left|\omega_{n}'(s)\right|\, ds
\leqslant 1+MT, \quad \forall t\in[0,T].
\end{align*}
Hence, the functions $\omega_n$  and their first derivative are uniformly bounded on $[0,T]$. For all $n\geq1$, $\theta_n=\omega_n\sigma_n$ is bounded as a product of two bounded functions since $\sigma_n$ satisfies \eqref{BurgersSourceTermEntropCond} and $u_r(t)$, $u_l(t)$ are bounded. Furthermore, $\theta_{n}'(t)$ is also bounded since
\begin{align*}
\left|\dfrac{d\theta_n}{dt}(t)\right|&=\left|b(t)\sigma_{n}(t)+\mu\big(u_a-\sigma_{n}(t)\big)\omega_{n}(t)-c(t)\right|\\
%&\leqslant \left|b(t)\right|\left|u_l(t)\right|+\mu\vert\omega_n(t)\vert\big(\left|u_a\right|+\left|u_l(t)\right|\big)+\left|c(t)\right|\\
&\leqslant \left|b(t)\right|\left|u_l(t)\right|+\mu\big(1+MT\big)\big(\left|u_a\right|+\left|u_l(t)\right|\big)+\left|c(t)\right|\\
&\leqslant Q=\max_{t\in[0,T]}\Big\lbrace \left|b(t)\right|\left|u_l(t)\right|+\mu\big(1+MT\big)\big(\left|u_a\right|+\left|u_l(t)\right|\big)+\vert c(t)\vert\Big\rbrace<\infty.
\end{align*}
Hence, $F_1$ is a family of bounded and equicontinuous functions at every point of the interval $[0,T]$.  Arzel\`a-Ascoli theorem ensures the existence of a subsequence $\lbrace(\omega_{n_k},\theta_{n_k})\rbrace \subset F_1$ that  converges uniformly to the continuous functions $(\omega,\theta)$ on the interval $[0,T]$. Since $\omega_{n_k}$ satisfies \eqref{OmgIncr} then $\omega$ is positive on any interval of the form $[\eta,T]$, with $0<\eta\leqslant T$.
Define the sequence of functions
\begin{equation}
\sigma_{n_k}(t)=\dfrac{\theta_{n_k(t)}}{\omega_{n_k}(t)}, \quad \forall t\in[0,T],
\label{Seq}
\end{equation}
and the function
\begin{equation}
\sigma(t)=\dfrac{\theta(t)}{\omega(t)},\quad \forall t\in[\eta,T].
\label{ConsSolGen}
\end{equation}
Clearly, $\sigma$ is a continuous on the interval $[\eta,T]$ as a quotient of two continuous functions. Let us prove that $\sigma_{n_k}$  converges uniformly to $\sigma$ on the interval $[\eta,T]$.  Let $t\in[\eta,T]$. We have:
\begin{equation}
\begin{aligned}
\left|\sigma_{n_k}(t)-\sigma(t)\right|&\leqslant \left|\frac{\theta_{n_k}(t)\omega(t)+\theta(t)\omega(t)-\theta(t)\omega(t)-\omega_{n_k}(t)\theta(t)}{\omega_{n_k}(\eta)\omega(\eta)}\right| \\
%&\leqslant \frac{{\mu}^2\vert\theta(t)\vert}{K^2(1-e^{-\mu\eta})^2}\left|\omega_{n_k}(t)-\omega(t)\right|+\frac{{\mu}^2\left|\omega(t)\right|}{K^2(1-e^{-\mu\eta})^2}\left|\theta_{n_k}(t)-\theta(t)\right| \\
&\leqslant \dfrac{{\mu}^2}{K^2(1-e^{-\mu\eta})^2}\Big(\vert\theta(t)\vert\left|\omega_{n_k}(t)-\omega(t)\right|+\left|\omega(t)\right|\left|\theta_{n_k}(t)-\theta(t)\right|\Big).\\
\end{aligned}
\label{helps2}
\end{equation}
Set $$P=\dfrac{\max_{t\in[\eta,T]}\big\lbrace \vert\theta(t)\vert,\vert\omega(t)\vert\big\rbrace {\mu}^2}{K^2(1-e^{-\mu\eta})^2}<\infty.$$
Taking the supremum in \eqref{helps2}, we get 
\begin{eqnarray}
\sup_{t\in[\eta,T]}\vert\sigma_{n_k}(t)-\sigma(t)\vert \leqslant P\big(\sup_{t\in[\eta,T]}\left|\omega_{n_k}(t)-\omega(t)\right|+\sup_{t\in[\eta,T]}\left|\theta_{n_k}(t)-\theta(t)\right|\big).
\label{UnifConv2}
\end{eqnarray}
Hence, the sequence $\sigma_{n_k}$ converges uniformly to $\sigma$ on the interval $[\eta, T]$ since $\omega_{n_k}$ and $\theta_{n_k}$  converge uniformly to $\omega$ and $\theta$, respectively, on the interval $[0, T]$. Let us prove that $(\omega,\sigma)$ is a $\mathcal{C}^1$-solution of  \eqref{GenRHCDropModSimp} on the interval $[\eta,T]$. For all $n_k\geq1$, we have:
\begin{equation}
\left\lbrace
\begin{aligned}
&\dfrac{d\omega_{n_k}}{dt}(t)=a(t)\sigma_{n_k}(t)-b(t),\\
&\dfrac{d(\omega_{n_k}\sigma_{n_k})}{dt}=b(t)\sigma_{n_k}(t)+\mu\big(u_a-\sigma_{n_k}(t)\big)\omega_{n_k}(t)- c(t).
\end{aligned}
\right.
\label{GenRHCSubSeq}
\end{equation}
As $\omega_{n_k}$ and $\sigma_{n_k}$ converge  uniformly  to $\omega$ and $\sigma$, respectively, on the interval $[\eta,T]$ then the terms on the r.h.s of each equation of \eqref{GenRHCSubSeq} also converge uniformly on the interval $[\eta,T]$, i.e.\ the sequence $\lbrace(\frac{d\omega_{n_k}}{dt},\frac{d(\omega_{n_k}\sigma_{n_k})}{dt})\rbrace_{k\geq1}$  converges uniformly on the interval $[\eta,T]$. Take $n_k\rightarrow \infty$ in \eqref{GenRHCSubSeq}. Then, by using theorem 7.17 in \cite{rudin1976principles}, we obtain 
\begin{equation}
\left\lbrace
\begin{aligned}
&\dfrac{d\omega}{dt}(t)=\lim_{n_k\rightarrow\infty}\dfrac{d\omega_{n_k}}{dt}(t)=\lim_{n_k\rightarrow\infty}\Big(a(t)\sigma_{n_k}(t)-b(t)\Big)=a(t)\sigma(t)-b(t),\\
&\dfrac{d(\omega\sigma)}{dt}(t)=\lim_{n_k\rightarrow\infty}\dfrac{d(\omega_{n_k}\sigma_{n_k})}{dt}(t)=\lim_{n_k\rightarrow\infty}\Big(b(t)\sigma_{n_k}(t)+\mu\big(u_a-\sigma_{n_k}(t)\big)\omega_{n_k}(t)-c(t)\Big)\\
&\qquad\qquad\qquad\qquad\quad\qquad\qquad\text{ }=b(t)\sigma(t)+\mu\big(u_a-\sigma(t)\big)\omega(t)-c(t).
\end{aligned}
\right.
\label{GenRHCLimit}
\end{equation} 
The derivative of the functions $\omega$ and $\theta$ are continuous on the interval $[\eta,T]$ since the terms on the r.h.s of \eqref{GenRHCLimit} are continuous.  Hence, the couple $(\omega,\sigma)$ is a $\mathcal{C}^1$-solution of \eqref{GenRHCDropModSimp} on the interval $[\eta,T]$. Since $\eta>0$ is arbitrary then the $\mathcal{C}^1$-solution $(\omega,\sigma)$ can be extended to all $t$ in $(0,T]$. From the uniform convergence of $\omega_{n_k}$ to $\omega$ on the interval $[0,T]$, we get 
\begin{equation}
\omega(0)=\displaystyle\lim_{n_k\rightarrow \infty}\omega_{n_k}(0)=\displaystyle\lim_{n_k\rightarrow \infty}\dfrac{1}{n_k}=0=\omega_0.
\end{equation}
We set $$\sigma(0)=\sigma_0\in(u_+,u_-).$$
Thus, the $\mathcal{C}^1$-functions $(\omega,\sigma)$  satisfy the initial value problem \eqref{GenRHCDropModSimp}-\eqref{GenRHCDropModSimpInitCond} for all $t\in[0,T]$. This solution can be extended to all $t\geqslant0$ since $T>0$ is arbitrary.  By Proposition \ref{DeltaShockEntropCond}, the shock speed $\sigma$ satisfies \eqref{BurgersSourceTermEntropCond} for all $t\geqslant 0.$ 
\end{proof}
\begin{rem}
The generalized Rankine-Hugoniot conditions \eqref{GenRHCDropModSimp}-\eqref{GenRHCDropModSimpInitCond} reduce to the classical Rankine-Hugoniot conditions for a contact discontinuity or a two-contact-discontinuity solution. In fact, if $u_-=u_+$ then $\omega=0$.  Hence,  \eqref{GenRHCDropModSimp} reduces to the classical Rankine-Hugoniot conditions given by
\begin{eqnarray}
\left\lbrace
\begin{aligned}
&\big(\alpha_r(t)-\alpha_l(t)\big)\sigma(t)-\big(\alpha_r(t)u_r(t)-\alpha_l(t)u_l(t)\big)=0,\qquad\qquad\qquad\qquad\qquad\qquad\qquad\\
&\big(\alpha_r(t)u_r(t)-\alpha_l(t)u_l(t)\big)\sigma(t)-\big(\alpha_r(t)u_r(t)^2-\alpha_l(t)u_l(t)^2\big)=0.
\end{aligned}
\right.
\end{eqnarray}
\end{rem}
\noindent
Proposition \ref{DeltaShockEntropCond} stipulates that the initial shock speed $\sigma_0$ should belong to the interval $(u_+,u_-)$. The exact value for $\sigma_0$ is given by the following result:
\begin{prop}
\label{SigEntrop}
Assume that $\alpha_l(0),\alpha_r(0)>0$. Let $(\omega,\sigma)\in\mathcal{C}^1(\mathbb{R}^+_0)\times\mathcal{C}^1(\mathbb{R}^+_0)$  be a solution to the GRH conditions \eqref{GenRHCDropModSimp}-\eqref{GenRHCDropModSimpInitCond}. If $\omega_0=0$ then the  initial shock speed $\sigma_0$ satisfying the GRH conditions and the Lax's entropy condition \eqref{BurgersSourceTermEntropCond} at the origin is given by
\begin{equation}
\sigma_0=\dfrac{\sqrt{\alpha_r(0)}u_r(0)+\sqrt{\alpha_l(0)}u_l(0)}{\sqrt{\alpha_r(0)}+\sqrt{\alpha_l(0)}}. 
\label{OmegaSigma0Opt}
\end{equation}
\end{prop}
\begin{proof}
By using the first equation of \eqref{GenRHCDropModSimp}, the second equation can be written as 
\begin{align}
\omega(t)\dfrac{d\sigma}{dt}(t)=&-a(t)\sigma(t)^2+2b(t)\sigma(t)-c(t) +\mu\big(u_a-\sigma(t)\big)\omega(t).
\label{Eq2}
\end{align}
 At $t=0$,  $\omega(0)=\omega_0=0$ and \eqref{Eq2} reduces to 
\begin{equation}
-a(0)\sigma(0)^2+2b(0)\sigma(0)-c(0)=0.
\label{SigmaQuadEqu}
\end{equation}
If $a(0)=0$ then \eqref{SigmaQuadEqu} has one solution given by
\begin{equation}
\sigma(0)=\dfrac{u_r(0)+u_l(0)}{2}\in(u_r(0),u_l(0))
\label{sig}
\end{equation}
that satisfies \eqref{BurgersSourceTermEntropCond} at the origin. If $a(0)\neq0$ then \eqref{SigmaQuadEqu} has two roots
\begin{align}
{\sigma(0)}_1=\dfrac{\sqrt{\alpha_r(0)}u_r(0)-\sqrt{\alpha_l(0)}u_l(0)}{\sqrt{\alpha_r(0)}-\sqrt{\alpha_l(0)}} \quad\text{and}\quad
{\sigma(0)}_2=\dfrac{\sqrt{\alpha_r(0)}u_r(0)+\sqrt{\alpha_l(0)}u_l(0)}{\sqrt{\alpha_r(0)}+\sqrt{\alpha_l(0)}}.
\label{root2}
\end{align} 
The first root ${\sigma(0)}_1$ does not always satisfy \eqref{BurgersSourceTermEntropCond}  at the origin. In fact,
\begin{equation}
\begin{aligned}
{\sigma_0}_1-u_l(0)=\dfrac{\sqrt{\alpha_r(0)}(u_r(0)-u_l(0))}{\sqrt{\alpha_r(0)}-\sqrt{\alpha_l(0)}}>0,\quad\text{if } \alpha_l(0)>\alpha_r(0).
\end{aligned}
\end{equation}
The second root ${\sigma(0)}_2\in (u_r(0),u_l(0))$ for all $\alpha_l(0),\alpha_r(0)>0$. Moreover, if $\alpha_l(0)=\alpha_r(0)$ then ${\sigma(0)}_2$ reduces to \eqref{sig}.
\end{proof}

For the solution of Riemann problem \eqref{DropModSimp} and \eqref{DropModInAlphaURPInitCond}, we proved the existence of a solution to the GRH conditions \eqref{GenRHCDropModSimp}-\eqref{GenRHCDropModSimpInitCond} satisfying the Lax's entropy  condition \eqref{BurgersSourceTermEntropCond}. In general, it might be hard to find the analytical solution  of \eqref{GenRHCDropModSimp}-\eqref{GenRHCDropModSimpInitCond}. At least for the Riemann problem, we are lucky to find an analytical solution of \eqref{GenRHCDropModSimp}-\eqref{GenRHCDropModSimpInitCond}  given by 
\begin{eqnarray}
\begin{aligned}
&\omega(t)=\omega_0+\dfrac{\sqrt{\alpha_-\alpha_+}(u_--u_+)}{\mu}\big(1-e^{-\mu t}\big),\\
&\sigma(t)=u_a+\left(\dfrac{\sqrt{\alpha_-}u_-+\sqrt{\alpha_+}u_+}{\sqrt{\alpha_-}+\sqrt{\alpha_+}}-u_a\right)e^{-\mu t}.
\end{aligned} 
\end{eqnarray}

\section{Numerical results}
\label{DropModSimpNumRes}
We perform some representative test cases to illustrate the theoretical analysis of delta-shock waves and vacuum states for the Eulerian droplet model. We employ  the Transport-Collapse method \cite{Bouchut2} to discretize the equations \eqref{DropModSimp} and we take  $u_a=1$.

\subsection{Riemann problem solution: Numerical solutions vs theoretical analysis}
We compute Riemann solutions for a delta-shock wave and a vacuum state. Recall that in real applications, the drag coefficient $\mu$ is not constant (depends on the Reynolds number of the particles). For physical applications with non-constant drag coefficient, we refer to \cite{Bourgault2,Bourgault4,Bourgault3,Gurris}. Here we assume $\mu=0.2$ is constant. For a delta-shock wave, we take the initial conditions
\begin{eqnarray}
\big(\alpha,u\big)(x,0)=
\left\lbrace
\begin{aligned}
&(0.008,1.5),\quad x\leqslant 0,\\
&(0.003,0.5),\quad x>0,
\end{aligned}
\right.
\label{DeltaShockInitCond}
\end{eqnarray}
which correspond to a physical case where initially the particles behind  move faster.  Theoretical  and numerical solutions are shown in \text{Figure} \ref{DeltaShockFig}. The particles behind catch those in front resulting in a huge concentration of particles leading the formation of point mass on the shock trajectory.
\begin{figure}[!h]
\centering
\begin{tabular}{cc}
\includegraphics[scale=0.42]{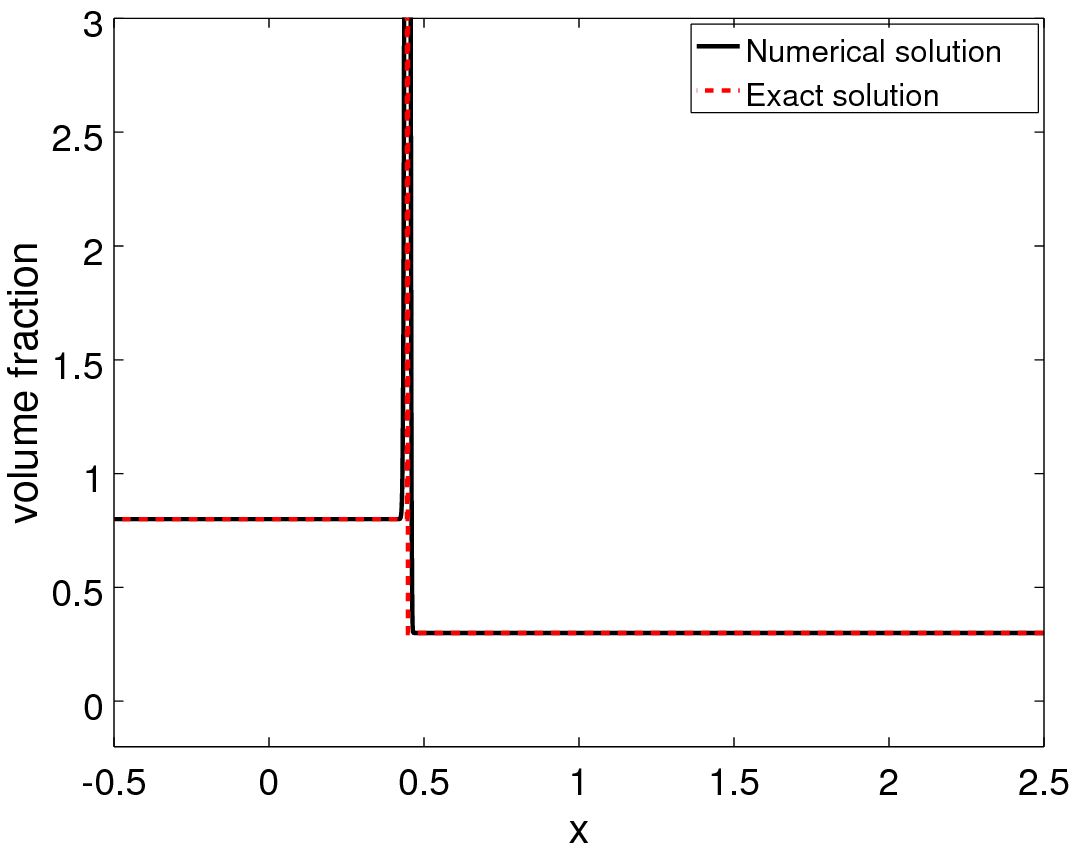}&\includegraphics[scale=0.42]{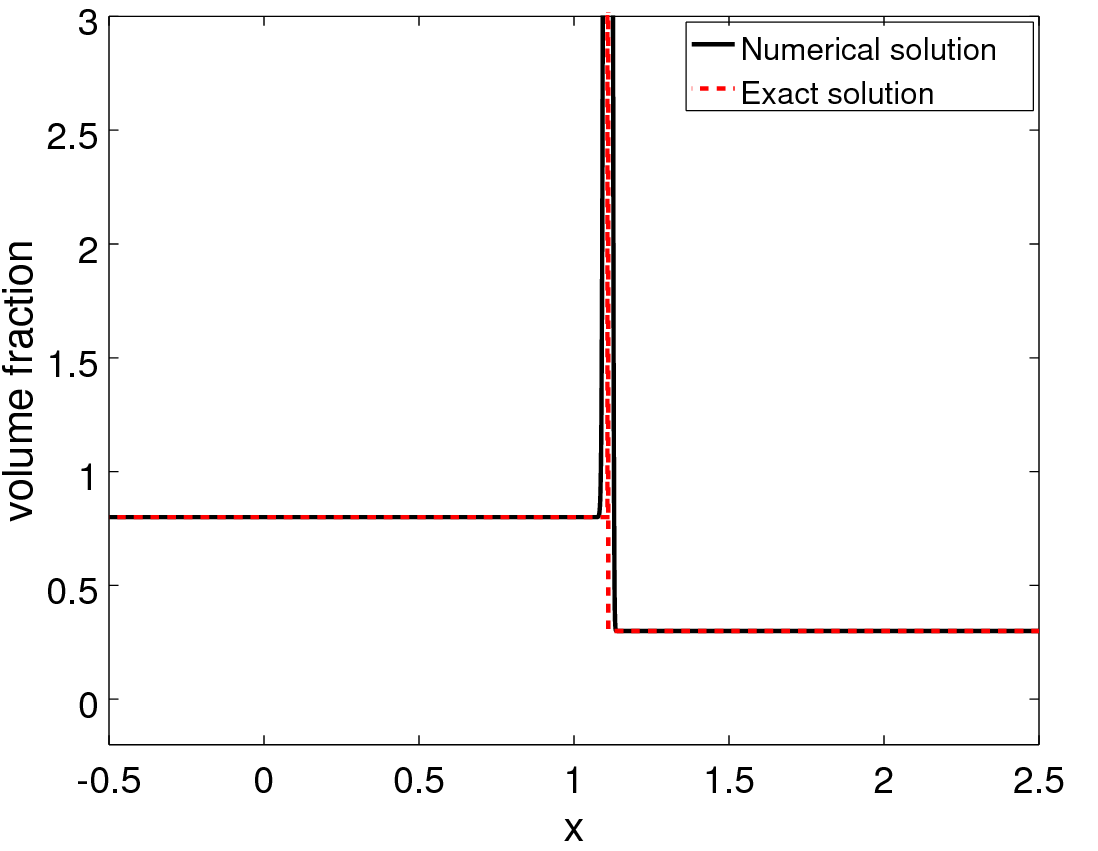}\\
\includegraphics[scale=0.42]{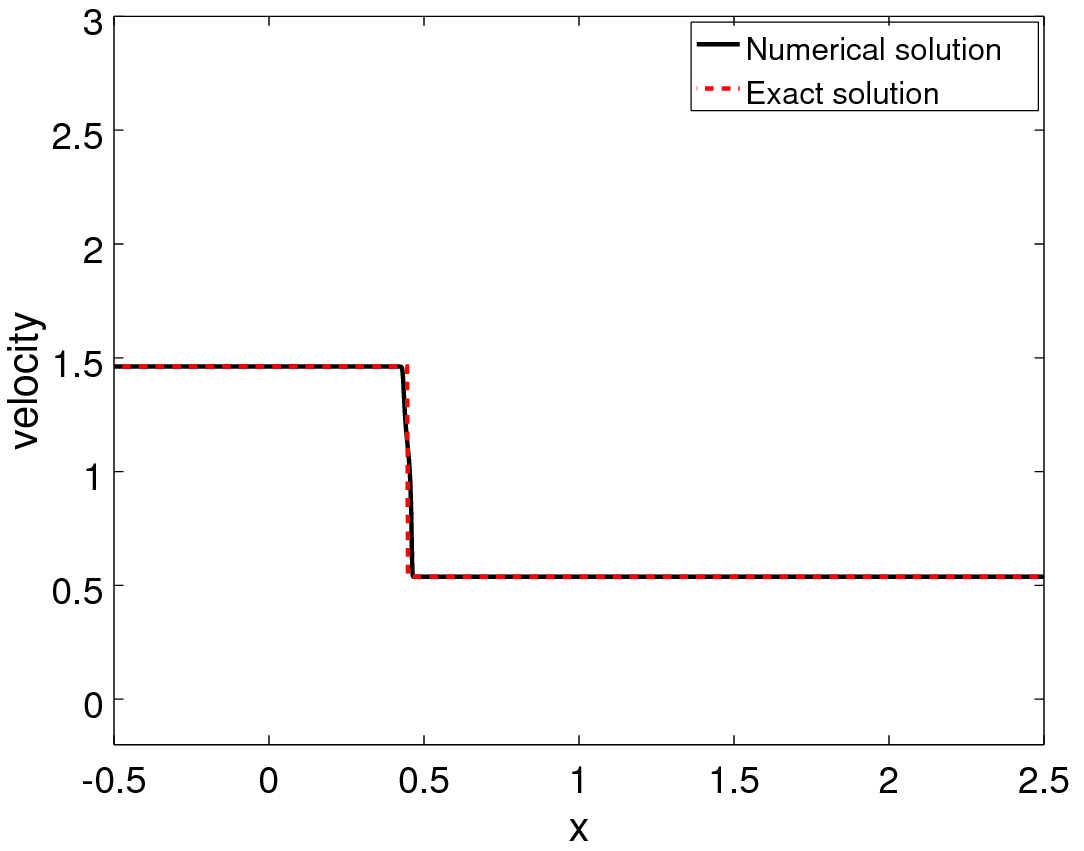}&\includegraphics[scale=0.42]{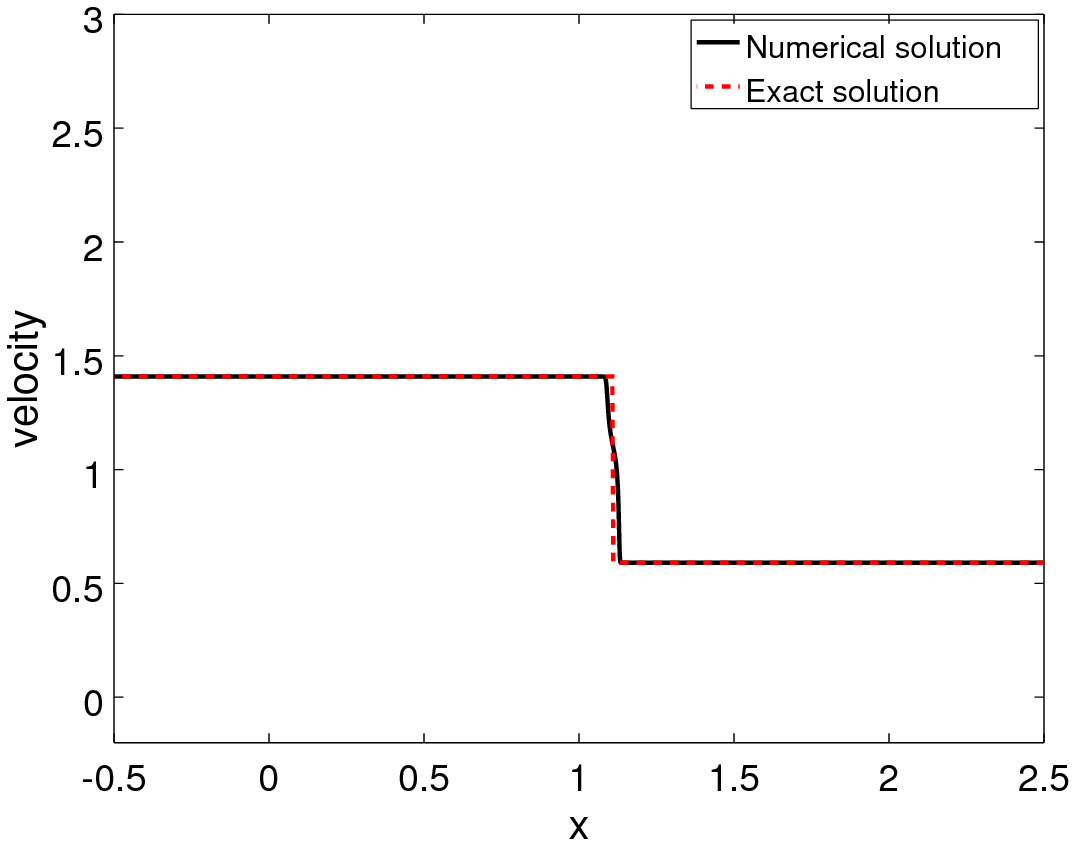}\\
$t=0.4$&$t=1$
\end{tabular}
\caption{A delta-shock wave of system \eqref{DropModSimp}: Exact and numerical solutions for two different times. $\mu=0.2$, $u_a=1$, $\Delta x=10^{-3}$ and $\Delta t=10^{-4}$.
\label{DeltaShockFig}}
\end{figure}

For a vacuum state, we take the initial conditions
\begin{eqnarray}
\big(\alpha,u\big)(x,0)=
\left\lbrace
\begin{aligned}
&(0.008,0.5),\quad x\leqslant 0,\\
&(0.003,1.5),\quad x>0,
\end{aligned}
\right.
\label{VacuumInitCond}
\end{eqnarray}
which correspond to a physical case where initially the particles in front move faster. Theoretical analysis and numerical solutions are represented in \text{Figure} \ref{VacuumFig}. We observe left and right nonvacuum states of particles delimiting a vacuum state, and moving at a continuous velocity. This theoretically corresponds to a two-contact-discontinuity with a vacuum state. 
\begin{figure}[!h]
\centering
\begin{tabular}{cc}
\includegraphics[scale=0.42]{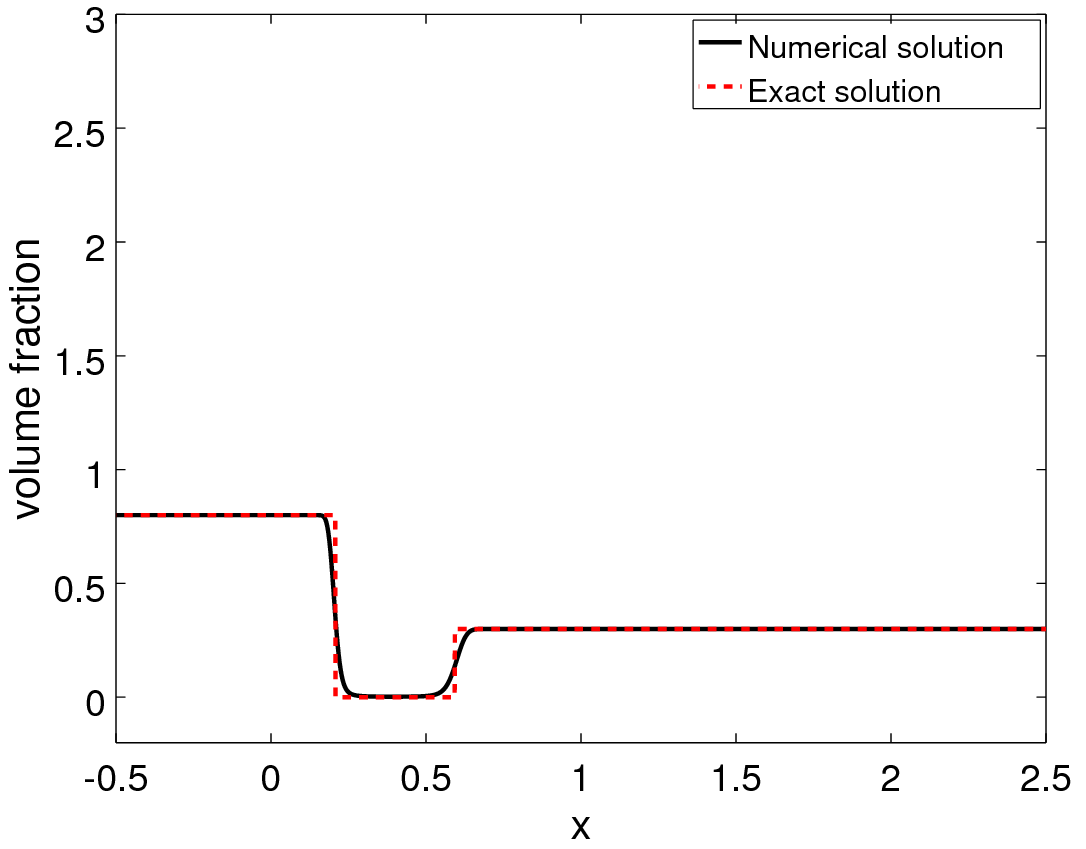}&\includegraphics[scale=0.42]{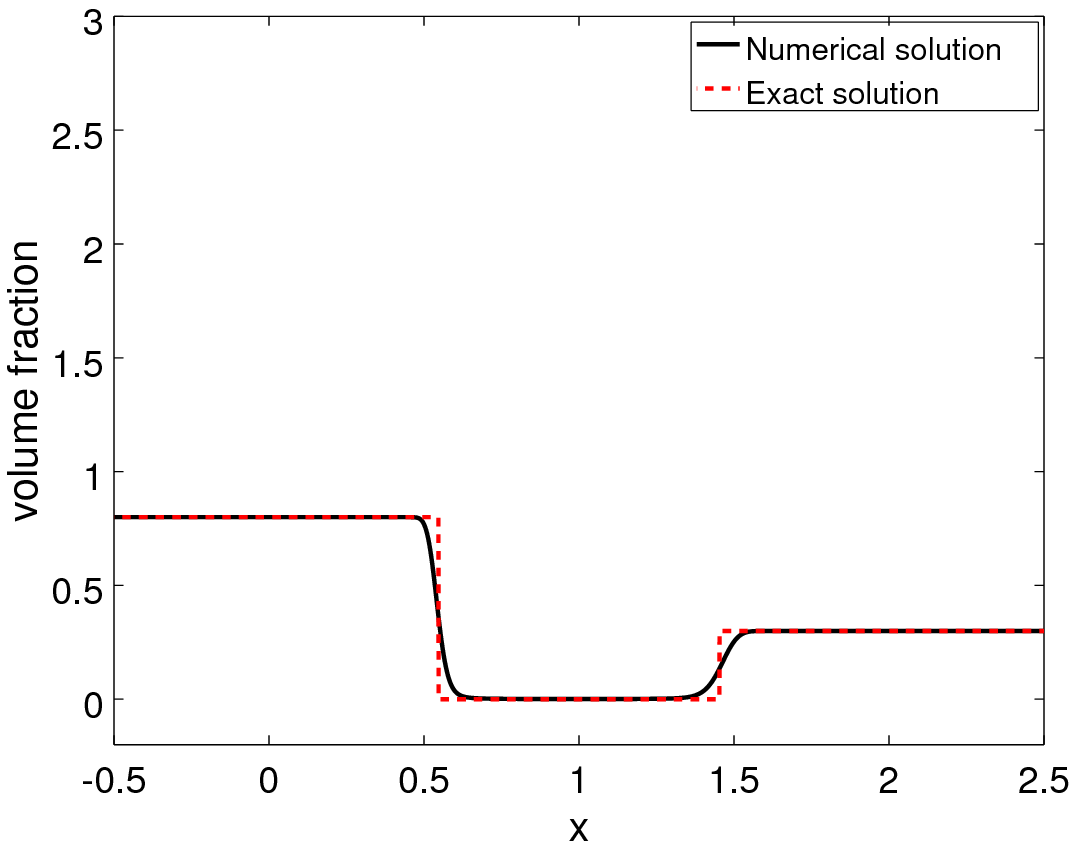}\\
\includegraphics[scale=0.42]{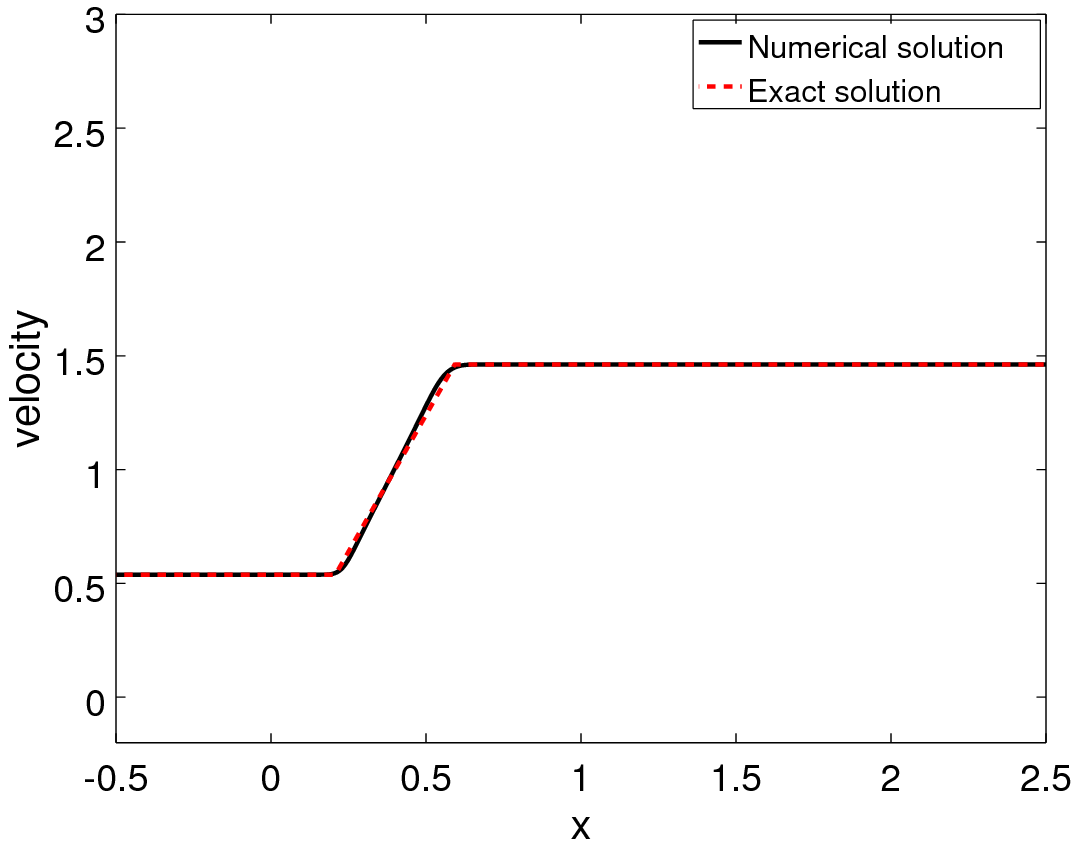}&\includegraphics[scale=0.42]{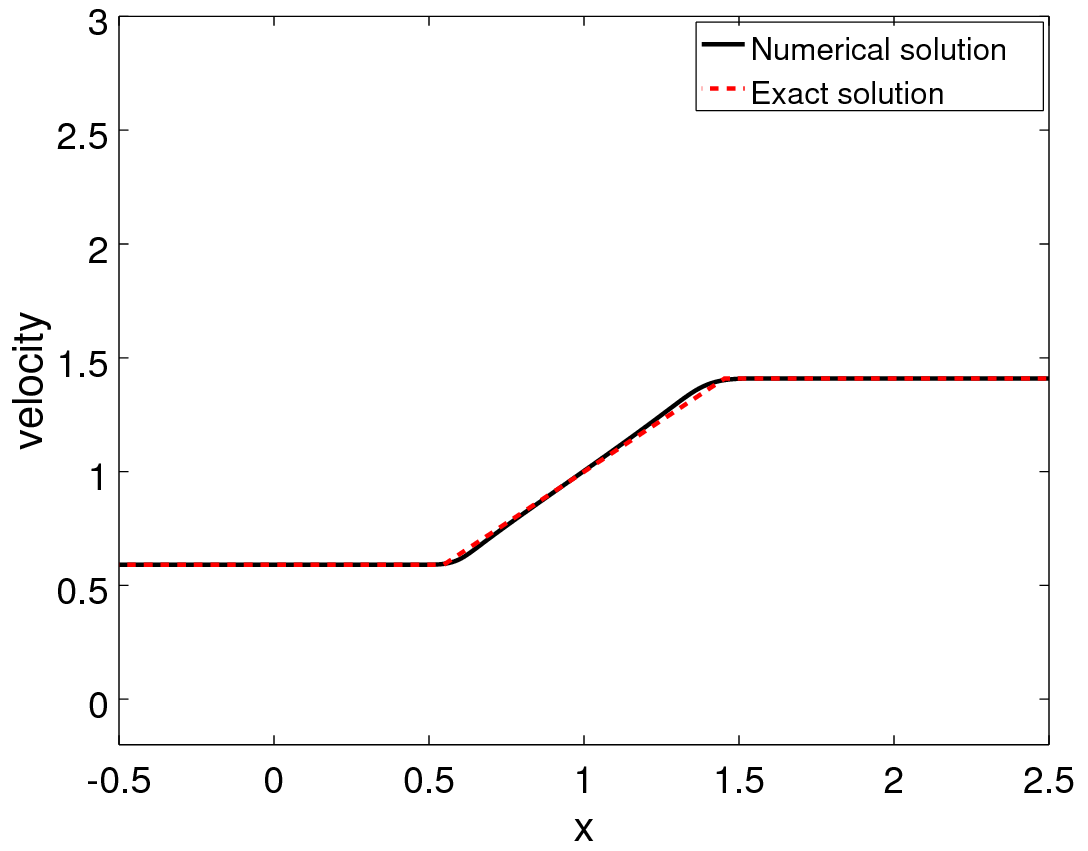}\\
$t=0.4$&$t=1$
\end{tabular}
\caption{A two-contact-discontinuity solution with a vacuum state of system \eqref{DropModSimp}: Exact and numerical solutions for two different times. $\mu=0.2$, $u_a=1$, $\Delta x=10^{-3}$ and $\Delta t=10^{-4}$.
\label{VacuumFig}}
\end{figure}
Note that in the above test cases and in the following, the displayed volume fraction is rescaled ($\times 100$).

In both cases the numerical results are in complete agreement with the theoretical analysis.

\subsection{Impact of the source term}
Recall that if $\mu=0$, i.e.\ there is no source term then system \eqref{DropModSimp} can be seen as the zero-pressure gas dynamics system whose  Riemann problem was solved in \cite{Sheng}.  We wish to highlight the impact of the zeroth order source term on the Riemann solution. We take $\mu=4$. Numerical results without ($\mu=0$) and with ($\mu=4$) the source term, computed with the initial conditions  \eqref{DeltaShockInitCond} and \eqref{VacuumInitCond}, are displayed in Figure~\ref{InfluenceSourceTermDeltaShockFig} and Figure~\ref{InfluenceSourceTermVacuumFig}, respectively. The solutions shown are obtained numerically, hence the delta-shocks can only have limited amplitude. The amplitude of the delta-shocks goes to infinity  as the mesh is refined. We notice that the zeroth order source term has significant impact on the solution. In fact, it acts as a relaxation term by weakening the delta-shock (seen as difference of amplitude in the numerical solution for the volume fraction) in Figure~\ref{InfluenceSourceTermDeltaShockFig}, and reducing the  extent of the vacuum region in Figure~\ref{InfluenceSourceTermVacuumFig}.  The left and right states of the velocity are no longer constant over time and tend to the air velocity which behaves as an equilibrium point. 
\begin{figure}[!h]
\centering
\begin{tabular}{cc}
\includegraphics[scale=0.45]{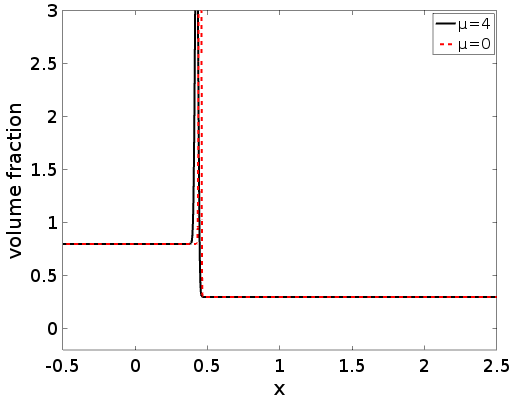} &\includegraphics[scale=0.45]{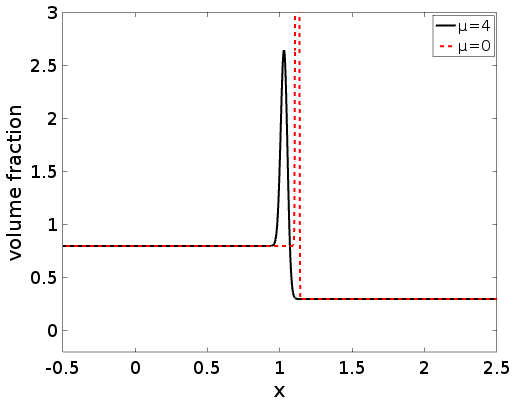}\\
\includegraphics[scale=0.45]{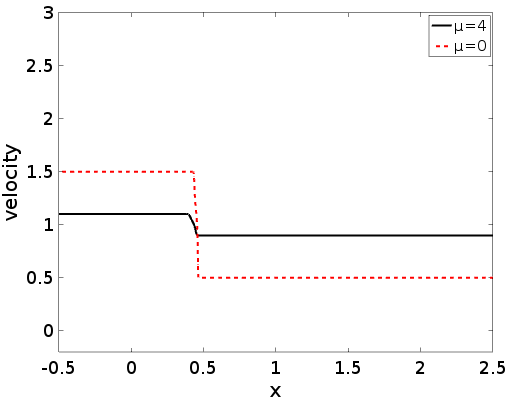} &\includegraphics[scale=0.45]{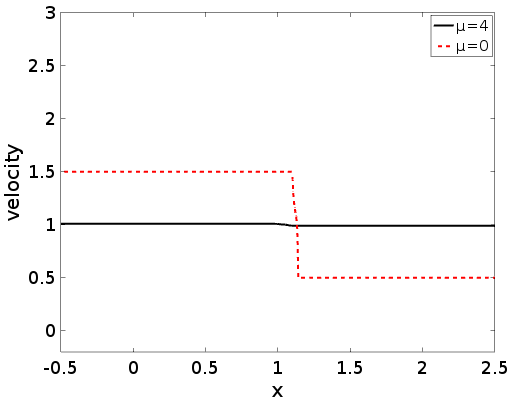}\\
$t=0.4$&$t=1$
\end{tabular}\\
\caption{Evolution of a delta-shock wave of system \eqref{DropModSimp}. Solutions  with ($\mu=4$)/without ($\mu=0$) the source term. $u_a=1$, $\Delta x=10^{-3}$ and $\Delta t=10^{-4}$.}
\label{InfluenceSourceTermDeltaShockFig}
\end{figure}
\begin{figure}[!h]
\centering
\begin{tabular}{cc}
\includegraphics[scale=0.45]{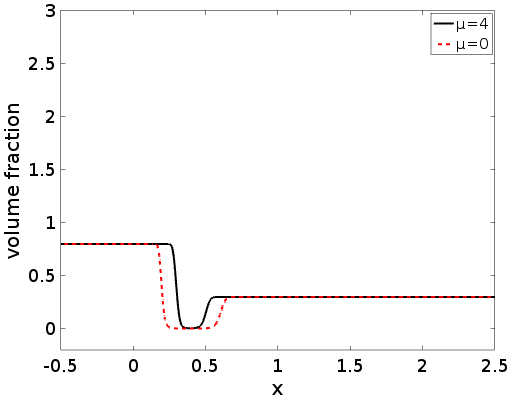} &\includegraphics[scale=0.45]{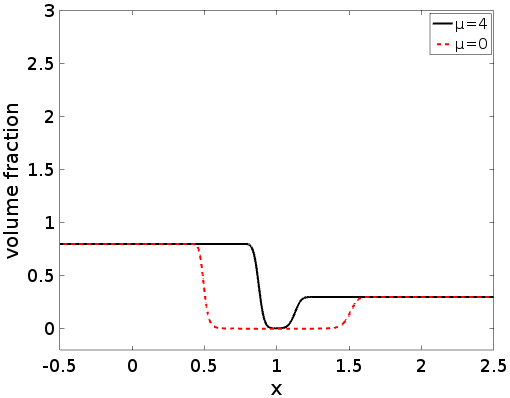}\\
\includegraphics[scale=0.45]{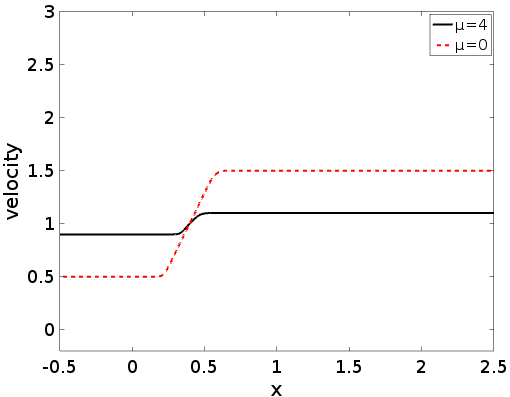} &\includegraphics[scale=0.45]{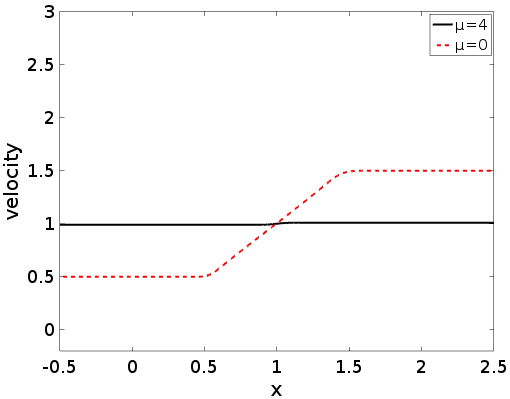}\\
$t=0.4$&$t=1$
\end{tabular}\\
\caption{Evolution of a vacuum state of system \eqref{DropModSimp}. Solutions  with ($\mu=4$)/without ($\mu=0$) the source term. $u_a=1$, $\Delta x=10^{-3}$ and $\Delta t=10^{-4}$.}
\label{InfluenceSourceTermVacuumFig}
\end{figure}

\section*{Acknowledgement}
This work was supported through a NSERC Discovery Grant.  The first author is a recipient of an ``International Admission Scholarship" and a  ``Bourse d'\'etudes pour la francophonie internationale" of the University of Ottawa for which we are deeply grateful.
\newpage
%\section*{References}
\bibliographystyle{plain} 
\bibliography{biblio}

\begin{thebibliography}{10}

\bibitem{Serge}
S.~Alinhac.
\newblock {\em Blowup for {N}onlinear {H}yperbolic {E}quations}.
\newblock Progress in Nonlinear Differential Equations and Their Applications.
  Birkh{\"a}user, Boston, 1995.

\bibitem{Drop2}
H.~Beaugendre, F.~Morency, and W.G. Habashi.
\newblock Developpement of a second generation in-flight icing simulation code.
\newblock {\em J. Fluids Eng.}, 128(2):378--387, 2006.

\bibitem{Bouchut2}
F.~Bouchut.
\newblock On zero pressure gas dynamics.
\newblock In {\em Advances in kinetic theory and computing}, volume~22 of {\em
  Ser. Adv. Math. Appl. Sci.}, pages 171--190. World Sci. Publ., River Edge,
  NJ, 1994.

\bibitem{Bouchut3}
F.~Bouchut and F.~James.
\newblock Duality solutions for pressureless gases, monotone scalar
  conservation laws, and uniqueness.
\newblock {\em Comm. {P}artial {D}ifferential {E}quations},
  24(11--12):2173--2189, 1999.

\bibitem{LBoudin}
L.~Boudin.
\newblock A solution with bounded expansion rate to the model of viscous
  pressureless gases.
\newblock {\em SIAM J. Math. Anal.}, 32(1):172--193, 2000.

\bibitem{Bourgault2}
Y.~Bourgault.
\newblock Computing gas-particle flows in airways with an {E}ulerian model.
\newblock In {\em European Conference on Computational Fluid Dynamics-ECCOMAS
  CFD 2006}, Egmond ann Zee, The Netherlands, 2006.

\bibitem{Bourgault1}
Y.~Bourgault, W.G. Habashi, J.~Dompierre, and G.S. Baruzzi.
\newblock A finite element method study of {E}ulerian droplets impingement
  models.
\newblock {\em Int. J. Numer. Methods Fluids}, 29(4):429--449, 1999.

\bibitem{Bourgault4}
Y.~Bourgault and M.~Thiriet.
\newblock Critical aspects of flow and aerosol simulations in the airway tract.
\newblock In {\em European Conference on Computational Fluid
  Dynamics--{E}{C}{C}{O}{M}{A}{S} {CFD} 2008}, Venice, Italy, 2008.

\bibitem{Bourgault3}
Z.~Boutanios, Y.~Bourgault, S.~Cober, G.A. Isaac, and W.G. Habashi.
\newblock {3D} droplets impingement analysis around an aircraft's nose and
  cockpit using {FENSAP-ICE}.
\newblock In {\em 36th {AIAA} {A}erospace {S}cience {M}eeting and {E}xhibit},
  Egmond ann Zee, The Netherlands, 1997.

\bibitem{Sticky}
Y.~Brenier and E.~Grenier.
\newblock Sticky particles and scalar conservation laws.
\newblock {\em SIAM J. Numer. Anal.}, 35(6):2317--2328, 1998.

\bibitem{ChangChouHongLin}
Y.~Chang, S.-W. Chou, J.M. Hong, and Y.-C. Lin.
\newblock Existence and uniqueness of {L}ax-type solutions to the {R}iemann
  problem of scalar balance law with singular source term.
\newblock {\em Taiwan. J. Math.}, 17(2):431--464, 2013.

\bibitem{Vacuum}
G.-Q. Chen and H.~Liu.
\newblock {Formation of delta-shocks and vacuum states in the vanishing
  pressure limit of solutions to the {E}uler equations for isentropic fluids}.
\newblock {\em SIAM J. Math. Anal.}, 34(4):925--938, 2003.

\bibitem{CHEN2004}
G.-Q. Chen and H.~Liu.
\newblock Concentration and cavitation in the vanishing pressure limit of
  solutions to the {E}uler equations for nonisentropic fluids.
\newblock {\em Phys. D}, 189(1--2):141--165, 2004.

\bibitem{HCheng}
H.~Cheng.
\newblock Delta shock waves for a linearly degenerate hyperbolic system of
  conservation laws of {K}eyfitz-{K}ranzer type.
\newblock {\em Adv. Math. Phys.}, 2013:1--10, 2013.

\bibitem{CHENG201117}
H.~Cheng and H.~Yang.
\newblock Riemann problem for the relativistic {C}haplygin {E}uler equations.
\newblock {\em J. Math. Anal. Appl.}, 381(1):17--26, 2011.

\bibitem{Cheng1997}
S.~Cheng, J.~Li, and T.~Zhang.
\newblock Explicit construction of measure solutions of {C}auchy problem for
  transportation equations.
\newblock {\em Sci. China Ser. A-Math.}, 40(12):1287--1299, 1997.

\bibitem{crouzeix1984analyse}
M.~Crouzeix and A.L. Mignot.
\newblock {\em Analyse {N}um{\'e}rique des {E}quations Diff{\'e}rentielles}.
\newblock Collection Math{\'e}matiques Appliqu{\'e}es pour la Ma{\^\i}trise.
  Masson, 1984.

\bibitem{DalMaso}
G.~Dal~Maso, P.~LeFloch, and F.~Murat.
\newblock Definition and weak stability of nonconservative products.
\newblock {\em J. Math. Pures Appl.}, 74:483--548, 1995.

\bibitem{DiPerna1983}
R.~DiPerna.
\newblock Convergence of the viscosity method for isentropic gas dynamics.
\newblock {\em Comm. Math. Phys.}, 91(1):1--30, 1983.

\bibitem{weinan1996}
W.~E, Yu.G. Rykov, and Ya.G. Sinai.
\newblock Generalized variational principles, global weak solutions and
  behavior with random initial data for systems of conservation laws arising in
  adhesion particle dynamics.
\newblock {\em Comm. Math. Phys.}, 177(2):349--380, 1996.

\bibitem{Evans2010}
L.C. Evans.
\newblock {\em Partial {D}ifferential {E}quations}.
\newblock Graduate Studies in Mathematics. American Mathematical Society, 2010.

\bibitem{FANG2012307}
B.~Fang, P.~Tang, and Ya-G. Wang.
\newblock The {R}iemann problem of the {B}urgers equation with a discontinuous
  source term.
\newblock {\em J. Math. Anal. Appl.}, 395(1):307 -- 335, 2012.

\bibitem{Raviart}
E.~Godlewski and P.-A. Raviart.
\newblock {\em Numerical {A}pproximation of {H}yperbolic {S}ystems of
  {C}onservation {L}aws}, volume 118 of {\em Applied Mathematical Sciences}.
\newblock Springer-Verlag, New York, 1996.

\bibitem{Gurris}
M.~Gurris, D.~Kuzmin, and S.~Turek.
\newblock Finite element simulation of compressible particle-laden gas flows.
\newblock {\em J. Comput. Appl. Math.}, 233(12):3121--3129, 2010.

\bibitem{Drop3}
W.G. Habashi.
\newblock Recent advances in {CFD} for in-flight icing simulation.
\newblock {\em J. Japan. Soc. Fluid. Mech.}, 28:99--118, 2009.

\bibitem{Hoff}
D.~Hoff and D.~Serre.
\newblock The failure of continuous dependence on initial data for the
  {N}avier-{S}tokes equations of compressible flow.
\newblock {\em SIAM J. Appl. Math.}, 51(4):887--898, 1991.

\bibitem{Huang2001}
F.~Huang and Z.~Wang.
\newblock Well posedness for pressureless flow.
\newblock {\em Comm. Math. Phys.}, 222(1):117--146, 2001.

\bibitem{KEYFITZ1995420}
B.L. Keyfitz and H.C. Kranzer.
\newblock Spaces of weighted measures for conservation laws with singular shock
  solutions.
\newblock {\em J. Diff. Equ.}, 118(2):420 -- 451, 1995.

\bibitem{Korchinski}
D.J. Korchinski.
\newblock {\em Solution of a {R}iemann problem for a {$2\times 2$} system of
  conservation laws possessing no classical weak solution}.
\newblock PhD thesis, Adelphi University, Garden City, New York, 1977.

\bibitem{Lax}
P.D. Lax.
\newblock {\em Hyperbolic {S}ystems of {C}onservation {L}aws and the
  {M}athematical {T}heory of {S}hock {W}aves}.
\newblock Society for Industrial and Applied Mathematics, 1973.

\bibitem{li2003}
J.~Li and G.~Warnecke.
\newblock Generalized characteristics and the uniqueness of entropy solutions
  to zero-pressure gas dynamics.
\newblock {\em Adv. Differential Equations}, 8(8):961--1004, 2003.

\bibitem{LIYANG}
J.~Li and H.~Yang.
\newblock Delta-shocks as limits of vanishing viscosity for multidimensional
  zero-pressure gas dynamics.
\newblock {\em Quart. Appl. Math.}, 59(2):315--342, 2001.

\bibitem{Zhang}
J.~Li, T.~Zhang, and S.~Yang.
\newblock {\em The {T}wo-{D}imensional {R}iemann {P}roblem in {G}as
  {D}ynamics}, volume~98 of {\em Pitman {M}onographs and {S}urveys in {P}ure
  and {A}pplied {M}athematics}.
\newblock Longman, Harlow, 1998.

\bibitem{LIN}
L.-W. Lin.
\newblock On the vacuum state for the equations of isentropic gas dynamics.
\newblock {\em J. Math. Anal. Appl.}, 121(2):406--425, 1987.

\bibitem{lions1996}
P.L. Lions.
\newblock {\em Mathematical {T}opics in {F}luid {M}echanics: Volumes 1-2}.
\newblock Oxford University Press, Incorporated, 1996,1998.

\bibitem{LIU1980}
T.-P. Liu and J.A. Smoller.
\newblock On the vacuum state for the isentropic gas dynamics equations.
\newblock {\em Adv. Appl. Math.}, 1(4):345--359, 1980.

\bibitem{mascia1997}
C.~Mascia and C.~Sinestrari.
\newblock The perturbed {R}iemann problem for a balance law.
\newblock {\em Adv. Differential Equations}, 2(5):779--810, 1997.

\bibitem{Drop1}
F.~Morency, H.~Beaugendre, and W.G. Habashi.
\newblock {FENSAP}-{ICE}: A comprehensive {3D} simulation tool for in-flight
  icing.
\newblock In {\em 15th AIAA Computational Fluid Dynamics Conference}, 2001.

\bibitem{rudin1976principles}
W.~Rudin.
\newblock {\em Principles of {M}athematical {A}nalysis}.
\newblock International series in Pure and Applied Mathematics. McGraw-Hill,
  1976.

\bibitem{Serre}
D.~Serre.
\newblock {\em Systems of {C}onservation {L}aws: {H}yperbolicity, {E}ntropies,
  {S}hock {W}aves}.
\newblock Cambridge University Press, Cambridge, 1999.

\bibitem{Shandarin}
S.F. Shandarin and Y.B. Zeldovich.
\newblock Large-scale structure of the universe: Turbulence, intermittency,
  structures in a self-gravitating medium.
\newblock {\em Rev. Mod. Phys.}, 61:185--220, 1989.

\bibitem{Sheng}
W.~Sheng and T.~Zhang.
\newblock {\em The {R}iemann {P}roblem for the {T}ransportation {E}quations in
  {G}as {D}ynamics}.
\newblock {A}merican {M}athematical {S}ociety: {M}emoirs of the {A}merican
  {M}athematical {S}ociety. American Mathematical Society, 1999.

\bibitem{SinestrariC}
C.~Sinestrari.
\newblock The {R}iemann problem for an inhomogeneous conservation law without
  convexity.
\newblock {\em SIAM J. Math. Anal.}, 28(1):109--135, 1997.

\bibitem{Smoller}
J.~Smoller.
\newblock {\em Shock {Wa}ves and {R}eaction-{D}iffusion {E}quations}.
\newblock Grundlehren der Mathematischen Wissenschaften. Springer-{V}erlag,
  {N}ew {Y}ork, 1994.

\bibitem{TAN19941}
D.~Tan, T.~Zhang, and Y.~Zheng.
\newblock Delta-shock waves as limits of vanishing viscosity for hyperbolic
  systems of conservation laws.
\newblock {\em J. Diff. Equ.}, 112(1):1--32, 1994.

\bibitem{WANG1997341}
Z.~Wang and X.~Ding.
\newblock Uniqueness of generalized solution for the {C}auchy problem of
  transportation equations.
\newblock {\em Acta Math. Scientia}, 17(3):341--352, 1997.

\bibitem{Whitham}
G.B. Whitham.
\newblock {\em Linear and {N}onlinear {W}aves}.
\newblock Pure and Applied Mathematics: A Wiley Series of Texts, Monographs and
  Tracts. Wiley, 2011.

\bibitem{YANG1999447}
H.~Yang.
\newblock Riemann problems for a class of coupled hyperbolic systems of
  conservation laws.
\newblock {\em J. Diff. Equ.}, 159(2):447 -- 484, 1999.

\bibitem{YANG2014}
H.~Yang and J.~Wang.
\newblock Delta-shocks and vacuum states in the vanishing pressure limit of
  solutions to the isentropic {E}uler equations for modified {C}haplygin gas.
\newblock {\em J. Math. Anal. Appl.}, 413(2):800--820, 2014.

\bibitem{YANG20125951}
H.~Yang and Y.~Zhang.
\newblock New developments of delta shock waves and its applications in systems
  of conservation laws.
\newblock {\em J. Diff. Equ.}, 252(11):5951 -- 5993, 2012.

\bibitem{YIN2009}
G.~Yin and W.~Sheng.
\newblock Delta shocks and vacuum states in vanishing pressure limits of
  solutions to the relativistic {E}uler equations for polytropic gases.
\newblock {\em J. Math. Anal. Appl.}, 355(2):594--605, 2009.

\bibitem{ZhangShen}
T.~Zhang and C.~Shen.
\newblock The shock wave solution to the {R}iemann problem for the {B}urgers
  equation with the linear forcing term.
\newblock {\em Appl. Anal.}, 95(2):283--302, 2016.

\end{thebibliography}

\end{document}